\documentclass[leqno]{siamart1116}

\usepackage{graphics,graphicx,epstopdf}

\usepackage{longtable}
\usepackage[tight]{subfigure}
\usepackage{float}
\usepackage{makecell}
\usepackage{comment}
\usepackage{mathtools}
\usepackage{listings} 

\usepackage[leqno]{amsmath}
\usepackage{amsxtra, amsfonts,amscd,amssymb,graphicx,epsf}

\usepackage[algo2e, vlined,ruled]{algorithm2e}

\numberwithin{equation}{section}

\definecolor{am}{RGB}{0,101,189}

\newtheorem{remark}[theorem]{Remark}

\newtheorem{assumption}[theorem]{Assumption}

\definecolor{mygreen}{rgb}{0,.5,0}

\newcommand{\be}{\begin{equation}}
\newcommand{\ee}{\end{equation}}
\newcommand{\bee}{\begin{equation*}}
\newcommand{\eee}{\end{equation*}}
\newcommand{\bea}{\begin{eqnarray}}
\newcommand{\eea}{\end{eqnarray}}
\newcommand{\beaa}{\begin{eqnarray*}}
\newcommand{\eeaa}{\end{eqnarray*}}
\newcommand{\R}{\mathbb{R}}
\newcommand{\C}{\mathbb{C}}

\newcommand{\tr}{\mathrm{tr}}

\newcommand{\iprod}[2]{ \left\langle #1, #2 \right\rangle}
\newcommand{\half}{\frac{1}{2}}

\newcommand{\sym}{{\mathrm{sym}}}
\newcommand{\proj}{{\mathrm{Proj}}}
\newcommand{\dist}{\mathrm{dist}}
\newcommand{\brev}[1]{{\color{black}{#1}}}

\newcommand{\diag}{\mathrm{diag}}

\newcommand{\N}{{\mathbb{N}}}
\newcommand{\conv}{\mathrm{conv}}
\newcommand{\kg}{\kappa_g}
\newcommand{\kH}{\kappa_H}
\newcommand{\sspan}{\mathrm{span}}

\newcommand{\st}{\mathrm{s.t.}}
\newcommand{\Fsf}{\mathsf{F}}
\newcommand{\Diag}{\mathrm{Diag}}

\newcommand{\grad}{\mathrm{grad}\;\!}
\newcommand{\Hess}{\mathrm{Hess}\;\!}

\newcommand{\Hc}{{\mathcal{H}^{\mathrm{c}}}}
\newcommand{\He}{{\mathcal{H}^{\mathrm{e}}}}

\newcommand{\hf}{{\mathrm{hf}}}
\newcommand{\ks}{{\mathrm{ks}}}
\newcommand{\f}{{\mathrm{f}}}
\newcommand{\ion}{{\mathrm{ion}}}
\newcommand{\xc}{{\mathrm{xc}}}
\newcommand{\dvec}{\mathop{\mathrm{vec}}}
\newcommand{\dmat}{\mathop{\mathrm{mat}}}

\newcommand{\nn}{\nonumber}

\newcommand{\Bcal}{{\mathcal{B}}}
\newcommand{\Rcal}{{\mathcal{R}}}
\newcommand{\Ecal}{{\mathcal{E}}}
\newcommand{\Sfrak}{{\mathfrak{S}}}

\newcommand{\Vcal}{{\mathcal{V}}}

\newcommand{\argmin}{\mathop{\mathrm{arg\, min}}}

\begin{document}
\title{Structured Quasi-Newton Methods for Optimization with Orthogonality Constraints}

\author{
	Jiang Hu\thanks{Beijing International Center for Mathematical Research, Peking University, China
		(\email{jianghu@pku.edu.cn}).}
	\and
	Bo Jiang\thanks{School of Mathematical  Sciences,  Key Laboratory  for	NSLSCS  of  Jiangsu  Province,  Nanjing  Normal  University,  China (\email{jiangbo@njnu.edu.cn}). Research supported in part by NSFC grants 11501298 and 11671036, and by the NSF of Jiangsu Province (BK20150965).}
	\and
	Lin Lin\thanks{Department of Mathematics, University of California, Berkeley, Berkeley, CA 94720 and
Computational Research Division, Lawrence Berkeley National Laboratory, Berkeley, CA 94720
		(\email{linlin@math.berkeley.edu}). Research supported in part by the National Science
		Foundation under Grant No. DMS-1652330, the Department of Energy under Grants
		No. DE-SC0017867 and No. DE-AC02-05CH11231, and the SciDAC project.}
	\and
	Zaiwen Wen\thanks{Beijing International Center for Mathematical
		Research, Peking University, China (\email{wenzw@pku.edu.cn}).
		Research supported in part by NSFC grants 11831002, 11421101 and 91730302, and by the National Basic Research Project under grant 2015CB856002.}
	\and
	Yaxiang Yuan\thanks{State Key Laboratory of Scientific and Engineering Computing, Academy of Mathematics and Systems Science, Chinese Academy of Sciences,
		China (\email{yyx@lsec.cc.ac.cn}). Research supported in part by NSFC grants 11331012 and 11461161005.}
}
\maketitle

\begin{abstract}
	In this paper, we study structured quasi-Newton methods for optimization problems with orthogonality constraints. Note that the  Riemannian Hessian of the objective function requires both the Euclidean Hessian and the Euclidean gradient. In particular, we are interested in applications that the Euclidean Hessian itself consists of a computational cheap part and a significantly expensive part.  Our basic idea is to keep these parts of lower computational costs but substitute those parts of higher computational costs by the limited-memory quasi-Newton update. More specically, the part related to Euclidean gradient and the cheaper parts in the Euclidean Hessian are preserved.  The initial quasi-Newton matrix is further constructed from a limited-memory Nystr\"om approximation to the expensive part.  Consequently, our subproblems approximate the original objective function in the Euclidean space and preserve the orthogonality constraints without performing the so-called vector transports. When the subproblems are solved to sufficient accuracy, both global and local q-superlinear convergence can be established under mild conditions. Preliminary numerical experiments on the linear eigenvalue problem and the electronic structure calculation show the effectiveness of our method compared with the state-of-art algorithms.  \end{abstract}


\section{Introduction}
In this paper, we consider the optimization problem with orthogonality constraints:
\be \label{prob:gen} \min_{X \in \C^{n\times p}} \quad f(X) \quad \st \quad X^*X = I_p, \ee
where $f(X):  \C^{n\times p} \rightarrow \R$ is a $\R$-differentiable function \cite{Ken-CR-2009}. Although our proposed methods are applicable to a general function $f(X)$, we are in particular interested in the cases that the Euclidean Hessian $\nabla^2 f(X)$ takes a natural structure as 
\be \label{equ:f:hessian:structure}
\nabla^2 f(X) = \Hc(X) + \He(X),
\ee
where the computational cost of $\He(X)$ is much more expensive than that of
$\Hc(X)$. This situation occurs when $f$ is a summation of functions whose full
Hessian are expensive to be evaluated or even not accessible. A practical
example is the Hartree-Fock-like total energy minimization problem in electronic structure theory \cite{szabo2012modern,martin2004electronic}, where the computation cost
associated with the Fock exchange matrix is significantly larger than the cost of the remaining components.

There are extensive methods for solving \eqref{prob:gen} in the literature. By
{exploring} the geometry of the manifold (i.e.,  orthogonality
constraints), the Riemannian gradient descent, conjugate gradient (CG), Newton
and trust-region methods are proposed in
\cite{gabay1982minimizing,EdelmanAriasSmith1999,udriste1994convex,smith1994optimization,AbsilBakerGallivan2007,opt-manifold-book,
wen2013feasible}. Since the second-order information sometimes is not
available, the quasi-Newton type method serves as an
alternative method to guarantee the good convergence property. Different from
the Euclidean quasi-Newton method, the vector transport operation \cite{opt-manifold-book} is used to compare
tangent vectors in different tangent spaces. After obtaining a descent
direction, the so-called retraction provides a curvilinear search along
manifold. By adding some restrictions between differentiable retraction and
vector transport, a Riemannian Broyden-Fletcher-Goldfarb-Shanno (BFGS) method is presented
in \cite{qi2011numerical,Ring2012Optimization,seibert2013properties}. Due to the
requirement of differentiable retraction, the computational cost associated with the vector transport operation may be costly. To avoid this
disadvantage, authors in
\cite{huang2013optimization,huang2015riemannian,huang2015broyden,huang2018riemannian}
develop a new class of Riemannian BFGS methods, symmetric rank-one (SR1) and Broyden family methods. Moreover, a selection of Riemannian quasi-Newton methods has been implemented in the software package Manopt \cite{manopt} and ROPTLIB \cite{huang2016roptlib}.    

\subsection{Our contribution} Since the set of orthogonal matrices can be viewed
as the Stiefel manifold, the existing quasi-Newton methods focus on the construction of an approximation to the Riemannian Hessian $\Hess f(X)$:  
\be \label{eq:hess-original} \Hess f(X)[{\xi}] = \proj_{X}( \nabla^2 f(X)[{\xi}] - {\xi}\sym(X^* \nabla f(X))), \ee
where  {$\xi$ is any tangent vector in the tangent space $T_X := \{\xi \in
\C^{n \times p}: X^*\xi + \xi^*X = 0 \}$} and $\proj_{X}(Z):= Z - X\sym(X^*Z)$
is the projection of $Z$ onto the tangent space $T_X$ and $\sym(A) := (A + A^*)/2$.
See \cite{absil2013extrinsic} for details on the structure
\eqref{eq:hess-original}. We briefly summarize our contributions as follows.  
\begin{itemize}
  \item  By taking the advantage of this structure
\eqref{eq:hess-original}, we construct an approximation to Euclidean Hessian
$\nabla^2 f(X)$ instead of the full Riemannian Hessian $\Hess f(X)$ directly, but keep the
remaining parts $\xi \sym(X^* \nabla f(X))$ and $\proj_X(\cdot)$. Then, we solve a subproblem {with} orthogonality constraints, whose objective function uses an approximate second-order Taylor expansion of $f$ with an extra regularization term. Similar to \cite{hu2018adaptive}, the trust-region-like strategy for the update of the regularization parameter and the modified CG method for solving the subproblem are utilized. 
      The vector transport is not needed in
      since we are working in the ambient Euclidean space. 
	\item By further taking advantage of the structure
      \eqref{equ:f:hessian:structure} of $f$, we develop a structured quasi-Newton approach to
      construct an approximation to the expensive part $\He$ while
      {preserving} the cheap part $\Hc$. This kind of structured
      approximation usually yields a better property than the approximation
      constructed by the vanilla quasi-Newton method. For the construction of an initial approximation of $\He$, we also investigate a limited-memory Nystr\"om approximation, which gives a subspace approximation of a known good but still complicated approximation of $\He$. 
    \item  When the subproblems are solved to certain
    accuracy, both global and local q-superlinear convergence can be established
    under certain mild conditions.
\item Applications to the linear eigenvalue problem and the electronic structure calculation are presented. The proposed algorithms perform comparably well with state-of-art methods in these two applications. 
\end{itemize}
\subsection{Applications to electronic structure calculation}
Electronic structure theories, and particularly Kohn-Sham density functional theory (KSDFT), play an important role in quantum physics, quantum chemistry and materials science. This problem can be interpreted as a minimization problem for the electronic total energy over multiple electron wave functions which are orthogonal to each other. The mathematical structure of Kohn-Sham equations depends heavily on the choice of the exchange-correlation (XC) functional.  With some abuse of terminology, throughout the paper we will use KSDFT to refer to Kohn-Sham equations with local or semi-local exchange-correlation functionals.  Before discretization, the corresponding Kohn-Sham Hamiltonian is a differential operator. On the other hand, when hybrid exchange-correlation functionals \cite{becke1993density,heyd2003hybrid} are used, the Kohn-Sham Hamiltonian becomes an integro-differential operator, and the Kohn-Sham equations become Hartree-Fock-like equations.  Again with some abuse of terminology, we will refer to such calculations as the HF calculation.  

For KSDFT calculations, the most popular numerical scheme is the self-consistent
field (SCF) iteration 
 which can be  efficient when combined with certain charge mixing techniques.
Since the hybrid exchange-correlation functionals depend on all the elements of the density matrix, HF
calculations are usually more difficult than KSDFT calculations. One commonly used algorithm is called the nested two-level SCF method \cite{giannozzi2009quantum}.  
In the inner SCF loop, by fixing the density matrix and the hybrid exchange
operator, it only performs an update on the charge density $\rho$, which is
solved by the SCF iteration. Once the stopping criterion of the inner iteration
is satisfied, the density matrix is updated in the outer loop according to the
Kohn-Sham orbitals computed in the inner loop. This method can also utilize the
charge mixing schemes for the inner SCF loop to accelerate convergence.
Recently, by combining with the adaptively compressed exchange operator (ACE)
method \cite{lin2016adaptively}, the convergence {rate} of {the}
nested two-level SCF method is greatly improved. Another popular algorithm to solve
HF calculations is the commutator direction inversion of the iterative subspace
(C-DIIS) method. By storing the density matrix explicitly, it can often lead to accelerated convergence rate. However, when the size of the density matrix becomes large, the storage cost of the density matrix becomes prohibitively expensive. Thus Lin et al. \cite{hu2017projected} proposed the projected C-DIIS (PC-DIIS) method, which only requires storage of wave function type objects instead of the whole density matrix. 

HF calculations can be also solved via using {the aforementioned Riemannian optimization methods}  (e.g., a feasible gradient method on the Stiefel manifold \cite{wen2013feasible}) without storing the density matrix or the wave function. 
 However, these existing methods often do not use the structure of the Hessian
 in KSDFT or HF calculations. In this paper,  by exploiting the structure of the
 Hessian,  we apply our structured quasi-Newton method to solve these problems.
 Preliminary numerical experiments show that our algorithm  performs at least
 comparably well with state-of-art methods in their convergent case. In the case
 that state-of-art methods failed, our algorithm often returns high quality solutions.

\subsection{Organization}
This paper is organized as follows. In section \ref{section:SQN}, we introduce
our structured quasi-Newton method and present our algorithm. In section
\ref{sec:convergenceanalysis}, the global and local convergence is analyzed
under certain inexact conditions. In sections \ref{section:EIG}  and
\ref{section:ks:hf}, detailed applications to the linear eigenvalue problem and the electronic structure calculation are discussed. Finally, we demonstrate the efficiency of our proposed algorithm in section \ref{section:NUM}. 

\subsection{Notation} 
For a matrix $X \in \C^{n \times p}$, we use $\bar X$, $X^*$,  $\Re X$ and $\Im X$  to denote its complex conjugate, complex conjugate transpose, real and imaginary parts, respectively.    Let $\sspan\{X_1, \ldots, X_l\}$ be the space spanned by the matrices $X_1, \ldots, X_l$. {The vector denoted $\dvec(X)$ in $\C^{np}$ is formulated by stacking each column of $X$ one by one, from the first to the last column; the operator $\dmat(\cdot)$ is the inverse of $\dvec(\cdot)$, i.e., $\dmat(\dvec(X)) = X$.} Given two matrices $A, B \in \C^{n \times p}$, the Frobenius inner product  $\langle \cdot, \cdot\rangle$ is  defined as  {$\langle A,B \rangle = \tr(A^*B)$} and the corresponding Frobenius norm $\|\cdot\|_{{\Fsf}}$ is  defined as $\|A\|_{{\Fsf}} = \sqrt{\tr(A^*A)}$. The Hadamard product of $A$ and $B$  is $A  \odot B$ with $(A \odot B)_{ij} = A_{ij} B_{ij}$. For a matrix ${M} \in \C^{n \times n}$,  the operator  {$\diag(M)$} is a vector in $\C^n$ formulated by the main diagonal of {$M$}; and for ${c} \in \C^n$,  the operator $\Diag({c})$ is an $n$-by-$n$ diagonal matrix with the elements of ${c}$ on the main diagonal. {The notation $I_p$ denotes the $p$-by-$p$ identity matrix.} Let $\mathrm{St}(n,p):=\{X \in \C^{n\times p}: X^*X = I_p \}$ be the {(complex)} Stiefel manifold.  {The notation $\N$ refers to  the  set of all natural numbers}. 

\section{A structured quasi-Newton approach}  \label{section:SQN}
\subsection{Structured quasi-Newton subproblem} In this subsection, we develop the structured quasi-Newton subproblem for solving \eqref{prob:gen}.  
 Based on the assumption \eqref{equ:f:hessian:structure}, methods using the
 exact Hessian $\nabla^2 f(X)$ may not be the best choices. When the
 computational cost of the gradient $\nabla f(X)$ is significantly cheaper than
 that of the Hessian $\nabla^2 f(X)$,  the quasi-Newton methods, which mainly
 use the gradients $\nabla f(X)$ to construct an approximation to $\nabla^2
 f(X)$,  may outperform other methods.
 Considering the form \eqref{equ:f:hessian:structure}, we can construct a
 structured quasi-Newton approximation {$\Bcal^k$} for $\nabla^2 f(X^k)$. The
 details will be presented in  section \ref{section:construction:Bk}. Note that a similar idea has been presented in \cite{zhou2010global} for the unconstrained nonlinear least square problems \cite{kass1990nonlinear,sun2006optimization}. 
Then our subproblem at the $k$-th iteration is constructed as 
\be \label{prob:struct-QN} 
\min_{X \in \C^{n\times p}} \  m_k(X)  \quad  \st \quad X^*X = I, 
\ee
where 
\[ m_k(X):= {\Re}\iprod{\nabla f(X^k)}{X- X^k} + \half {\Re}\iprod{\Bcal^k[X-X^k]}{X-X^k} + \frac{\tau_k}{2} d(X, X^k) \]
is an approximation to $f(X)$ in the Euclidean space. For the second-order Taylor
expansion of $f(X)$ at a point $X^k$, we refer to \cite[section 1.1]{wen2013adaptive} for details. Here,
$\tau_k$ is a regularization parameter and $d(X,X^k)$ is a proximal term to guarantee the convergence.

The proximal term can be chosen as the quadratic regularization 
\be \label{eq:quad-dist-stiefel} d(X,X^k) = \|X-X^k\|_{{\Fsf}}^2 \ee
or the cubic regularization
\be \label{eq:cubic-dist-stiefel}  d(X, X^k) = \frac{2}{3} \| X - X^k \|_{{\Fsf}}^3. \ee 
In the following, we will mainly focus on the quadratic regularization \eqref{eq:quad-dist-stiefel}. Due to the Stiefel manifold constraint, the quadratic regularization \eqref{eq:quad-dist-stiefel}  is actually equivalent to the linear term {$-2\Re\iprod{X}{X^k}$}.  {By using} the Riemannian Hessian {formulation \eqref{eq:hess-original}} on the Stiefel manifold, we have
\be \label{eq:hess-stiefel} \Hess m_k(X^k)[{\xi}] = \proj_{X^k}\left({\Bcal^k}[{\xi}] - {\xi}\sym((X^k)^* \nabla f(X^k)\right) + \tau_k {\xi}, ~ {\xi} \in T_{X^k}. \ee
Hence, the regularization term is to shift the spectrum of the corresponding Riemannian Hessian of the approximation {$\Bcal^k$} with $\tau_k$.

The Riemannian quasi-Newton methods for \eqref{prob:gen} in the literature
\cite{huang2016roptlib,huang2015riemannian,huang2016riemannian,huang2015broyden}
focus on constructing an approximation to the Riemannian Hessian $\Hess
f(X^k)$ directly without using its special structure
\eqref{eq:hess-original}. Therefore, vector transport needs to be utilized to
transport the tangent vectors from different tangent spaces to one common tangent space. 
If $p \ll n$, the second term $\sym\left((X^k)^*\nabla f(X^k)\right)$ is a small-scaled matrix and thus {can be computed with low cost}. In this case, after computing the approximation ${\Bcal^k}[{\xi}]$ of $\nabla^2 f(X) [\xi]$, we obtain a structured Riemannian quasi-Newton approximation $\proj_{X^k}\left({\Bcal^k}[{\xi}] - {\xi}\sym((X^k)^* \nabla f(X^k)\right)$ of $\Hess f(X^k)[{\xi}]$ without using any vector transport.

\subsection{Construction of {$\Bcal^k$}} \label{section:construction:Bk}
The classical quasi-Newton methods construct the approximation $\Bcal^k$ such that it satisfies the secant condition
     \be\label{eq:secant:0} 
     {\Bcal^k[S^k] = \nabla f(X^k) - \nabla f(X^{k-1}),}
     \ee
     {where}  $S^k:=X^k - X^{k-1}$. 
   Noticing {that $\nabla^2 f(X)$ takes the natural} structure
   \eqref{equ:f:hessian:structure},  it is reasonable to keep the cheaper part
   $\Hc(X)$ while only to approximate  $\He(X)$. Specifically,  we derive the approximation {$\Bcal^k$} to the Hessian $\nabla^2 f(X^k)$  as
    \be \label{eq:struct-qn} \Bcal^k = \Hc(X^k) + \Ecal^k, \ee
    where ${\Ecal^k}$ is an approximation to $\He(X^k)$.  Substituting
    \eqref{eq:struct-qn}  into \eqref{eq:secant:0}, we can see that the
    approximation $\Ecal^k$  should satisfy the following revised secant condition 
            \be \label{eq:secant} {\Ecal^k} [S^k] = Y^k, \ee 
    where 
    \be \label{eq:yk} Y^k := \nabla f(X^k) - \nabla f(X^{k-1}) - \Hc(X^k) {[S^k]}. \ee

    {For the large scale optimization problems, the limited-memory quasi-Newton methods are preferred since they often make simple but good approximations of the exact Hessian. Considering that the part $\He(X^k)$} itself may not be positive definite even when $X^k$ is  optimal, we utilize  the limited-memory symmetric rank-one (LSR1) scheme to approximate $\He(X^k)$ such that it satisfies the secant equation \eqref{eq:secant}.

Let $l = \min\{k,m\}$. We define the $(np) \times l$ matrices $S^{k,m}$ and $Y^{k,m}$ by 
\[S^{k,m} = \left[ \dvec(S^{k-l}), \ldots, \dvec(S^{k-1})\right], \quad Y^{k,m} = \left[ \dvec(Y^{k-l}), \ldots, \dvec(Y^{k-1})\right].\] Let $\Ecal_0^k: \C^{n \times p} \rightarrow \C^{n \times p}$ be the initial approximation of $\He(X^k)$ and define  the $(np) \times l$ matrix ${\Sigma}^{k,m} := \left[\dvec\!\left(\Ecal_0^k[S^{k-l}]\right), \ldots, \dvec\!\left(\Ecal_0^k[S^{k-1}]\right)\right]$. 
       Let $F^{k,m}$ be a matrix in $\C^{l \times l}$ with  $(F^{k,m})_{i,j} =  \iprod{S^{k-l + i - 1}}{Y^{k-l + j - 1}}$ for $1 \leq i,j \leq l$.  
Under the assumption  that $\iprod{S^{j}}{\Ecal^{j}[S^{j}] - Y^{j}} \neq 0$, $j = k-l, \ldots, k - 1$, it follows  from  \cite[Theorem 5.1]{byrd1994representations} that the matrix $F^{k,m} - (S^{k,m})^* \Sigma^{k,m}$ is invertible and the LSR1 gives 
    \be\label{l-sr1}
    \Ecal^k[U] = \Ecal_0^k[U] +   \dmat\left(N^{k,m} \left(F^{k,m} - (S^{k,m})^* \Sigma^{k,m}\right)^{-1} (N^{k,m})^*\dvec(U)\right), 
    \ee
    where $U \in \C^{n \times p}$ is any direction and $N^{k,m} = Y^{k,m} - \Sigma^{k,m}$. In the practical implementation, we  skip the update if 
    \[ \left | \iprod{S^{j}}{\Ecal^{j}[S^{j}] - Y^{j}} \right| \leq r\| S^j\|_\Fsf \| \Ecal^{j}[S^{j}] - Y^{j} \|_\Fsf \] 
    with small number $r$, say $r = 10^{-8}$. Similar idea can be found in \cite{NocedalWright06}.

      	\subsection{Limited-memory Nystr\"om  approximation {of $\Ecal_0^k$}}
      A good initial guess to the exact Hessian is also important to accelerate the convergence of the limited-memory quasi-Newton method. Here, we assume that a good initial approximation ${\Ecal_0^k}$ of the expensive part of the Hessian $\He(X^k)$ is known but {its computational cost is still very high}. We conduct how to use the {limited-memory Nystr\"om  approximation} to construct 
	another approximation with {lower} computational cost based on ${\Ecal_0^k}$. 
      	    
	    Specially, let $\Omega$ be a matrix whose columns form an orthogonal basis of a well-chosen subspace $\Sfrak$ and denote {$W = \Ecal_0^k [\Omega]$}. 
      	To reduce the computational cost and keep the good property of ${\Ecal_0^k}$, we construct the following approximation 
      	\be \label{eq:nystrom} 
	   {\hat{\Ecal}_0^k [U]} \coloneqq W(W^*\Omega)^\dag W^* {U}, \ee
	  {where $U \in \C^{n \times p}$ is any direction}.
      This is called {the limited-memory} Nystr\"om approximation; {see}  \cite{tropp2017fixed} and references therein {for more details}. By choosing the dimension of the subspace $\Sfrak$ {properly}, the rank of  {$W(W^* \Omega)^\dag W^*$} can be small {enough such that the} computational cost {of $\hat{\Ecal}_0^k [U]$ is} significantly reduced. {Furthermore}, we still want ${\hat{\Ecal}_0^k}$ {to satisfy the secant condition \eqref{eq:secant} as ${\Ecal_0^k}$ does.}  {More specifically, we need to  seek the subspace $\Sfrak$ such that} the secant condition     
      	\[ {\hat{\Ecal}_0^k} [S^k] = Y^k\]
	holds. 
      	{To this aim}, the subspace $\Sfrak$ can be chosen as $$\mathrm{span} \{ X^{k-1}, X^k\},$$ which contains the element $S^k$. By assuming that $\Ecal_0^k[UV] = \Ecal_0^k[U]V$ for any matrices $U,V$ with proper dimension (this condition is satisfied when $\Ecal_0^k$ is a matrix), we have $\hat{\Ecal}_0^k$ will satisfy the secant condition whenever $\Ecal_0^k$ does. 
	From the methods for linear eigenvalue computation in \cite{knyazev2001toward} and \cite{liu2013limited}, the subspace {$\Sfrak$} can also be decided as 
      	\be \label{eq:aug-sub} \mathrm{span} \{X^{k-1},X^k, \Ecal_0^k [X^{k}] \} \quad \mathrm{{or}} \quad  \mathrm{span} \{ X^{k-h}, \ldots, X^{k-1}, X^k \}\ee
      	with small memory length $h$. Once the subspace is defined, we can obtain the {limited-memory} Nystr\"om approximation by computing the $\Ecal_0^k[\Omega]$ once and {the pseudo} inverse of a small {scale} matrix. 
  \subsection{A structured quasi-Newton method with subspace refinement}
  Based on the theory of quasi-Newton method for unconstrained optimization, we
  know that algorithms which set the solution of \eqref{prob:struct-QN} as the
  next iteration point may not converge if no proper requirements on
  approximation $\Bcal^k$ or the regularization parameter $\tau_k$. Hence, we update
  the regularization parameter here using a trust-region-like strategy.
  Refereeing to \cite{hu2018adaptive}, we compute a trial point $Z^k$ by utilizing a modified CG method to solve the subproblem inexactly, which is to solve the Newton equation of \eqref{prob:struct-QN} at $X^k$ {as}
  \be \label{eq:newtoneq} \grad m_k(X^k) + \Hess m_k(X^k)[\xi] = 0, \quad \xi \in T_{X^k} , \ee
  where $\grad m_k(X^k) = \grad f(X^k)$ and $\Hess m_k(X^k)$ are given in \eqref{eq:hess-stiefel}. 
   After obtaining {the} trial point $Z^k$ of \eqref{prob:struct-QN},  {we calculate} the ratio between the predicted reduction and the actual reduction
  \be \label{eq:ratio} r_k = \frac{ f(Z^k) - f(X^k)} {
   m_k(Z^k) }.\ee
    If $r_k \ge \eta_1 > 0$, then the iteration is successful and we set
   $X^{k+1}= Z^k$; otherwise, the iteration is unsuccessful and we set $X^{k+1}=
   X^k$, that is,
   \be \label{eq:update-x}
   X^{k+1} = \begin{cases} Z^k, & \mbox{ if } r_k \ge \eta_1, \\
     X^k, & \mbox{ otherwise}.
   \end{cases}
   \ee
   The regularization parameter $\tau_{k+1}$ is updated as
   \be \label{alg:tau-up} \tau_{k+1} \in \begin{cases} (0, \gamma_0 \tau_k], & \mbox{  if }
     r_k \ge \eta_2, \\ [\tau_k,\gamma_1 \tau_k], & \mbox{  if } \eta_1 \leq r_k
     < \eta_2, \\ [\gamma_1 \tau_k, \gamma_2 \tau_k], & \mbox{  otherwise}{,} \end{cases} \ee
    where $0<\eta _1 \le \eta _2 <1 $ and $0 < \gamma_0 < 1<  \gamma _1 \le \gamma _2 $.   These parameters  determine how aggressively the
    regularization parameter is decreased when an iteration is successful or it is
    increased when an iteration is unsuccessful.   In practice, the performance of the
    regularized trust-region algorithm is not very sensitive to the  values of the parameters. 
    
    Noticing that the Newton-type method may still be very slow when the Hessian
    is close to be singular \cite{byrd2004convergence}. Numerically, it may
    happen that the regularization parameter turns to be huge  and the
    Riemannian Newton direction is nearly parallel to the negative gradient
    direction. Hence, it leads to an update $X^{k+1}$ belonging to the subspace
    $\tilde{\Sfrak}^k := \mathrm{span}\{X^k, \grad f(X^k)\}$, which is similar
    to the Riemannian gradient approach. To overcome this issue, we propose an optional step of solving \eqref{prob:gen} restricted to a subspace. Specifically, at $X^k$, we construct a
    subspace $\Sfrak^k$ with an orthogonal basis $Q^k \in \C^{{n\times q}} ({p \leq q
    \leq n})$, where $q$ is the dimension of $\Sfrak^k$. Then any point $X$ in
    the subspace $\Sfrak^k$ can be represented by 
    \[ X = Q^k M \]
    for some $ M \in \C^{{q \times p}}$.   
    Similar to the constructions of linear eigenvalue problems in \cite{knyazev2001toward} and \cite{liu2013limited}, the subspace $\Sfrak^k$ can be decided by using the history information $\{X^{k}, X^{k-1}, \ldots \}$, $\{ \grad f(X^{k}), \grad f(X^{k-1}),  \ldots \}$ and other useful information. 
    Given the subspace $\Sfrak^k$, the subspace method aims to find a solution of \eqref{prob:gen} with an extra constraint $X \in \Sfrak^k$, namely,
    \be \label{prob:subspace} 
    \min_{M \in \C^{{q \times p}}} \quad f(Q^kM) \quad \st \quad  M^*M = I_p.  \ee
    The problem \eqref{prob:subspace} can be solved {inexactly} by existing methods for optimization with orthogonality constraints.
    Once a good approximate solution $M^k$ of \eqref{prob:subspace} is obtained, then we {update} $X^{k+1} = Q^{k}M^{k}$ which is an approximate minimizer in the subspace $\Sfrak^k$ instead of $\tilde{\Sfrak}^k$. This 
    completes one step of the subspace iteration. 
    In fact, we
    compute the ratios between the norms of the Riemannian gradient of the
    last few iterations.  If all of these ratios are almost 1, we infer that
    the current iterates stagnates and the subspace method is called. 
    Consequently, our algorithm framework is outlined in Algorithm \ref{alg:struct-QN}.       
      \begin{algorithm2e}[!htbp]
      \caption{A structured quasi-Newton method with subspace refinement} 
      \label{alg:struct-QN}
      Input initial guess $X^0 \in \C^{n\times p}$ with $(X^0)^*X^0 = I_p$ and {the memory length $m$}.\\
      Choose $\tau_0>0$, $0<\eta _1 \le \eta _2 <1 $, $1<  \gamma _1 \le \gamma _2 $.  Set $k = 0$.
      \\
      \While{stopping conditions not met}
      {
      {Choose $\Ecal_0^k$ (use the limited-memory Nystr\"om  approximation if necessary)}.\\ 
      Construct the approximation {$\Bcal^k$} {via  \eqref{eq:struct-qn} and \eqref{l-sr1}}.\\
      {Construct} the subproblem \eqref{prob:struct-QN} and use the modified CG method (Algorithm 2 in \cite{hu2018adaptive}) to compute a new trial point $Z^k$. \\
      Compute the ratio $r_k$ via \eqref{eq:ratio}. \\
      Update $X^{k+1}$ from the trial point $Z^k$ based on \eqref{eq:update-x}. \\
      Update $\tau_{k}$ according to \eqref{alg:tau-up}. \\
      $k\gets k+1$.\\
      \If{{stagnate} conditions met}
      { Solve the subspace problem \eqref{prob:subspace} to update $X^{k+1}$.\\}
      }
    \end{algorithm2e}      
 \section{Convergence analysis} \label{sec:convergenceanalysis} 
 In this section, we present the convergence property of Algorithm
 \ref{alg:struct-QN}. 
 To guarantee the global convergence and fast local convergence rate, the inexact conditions for the subproblem \eqref{prob:struct-QN} (with quadratic or cubic regularization) can be chosen as 
 \bea \label{eq:inexact-1} m_k(Z^k)  & \leq -c\|\grad f(X^k)\|_{{\Fsf}}^2  \\
  \label{eq:inexact-2} \| \grad m_k(Z^k)\|_{{\Fsf}} & \leq \theta^k \| \grad f(X^k) \|_{{\Fsf}}\eea  
 with some positive constant $c$ and $\theta^k:= \min\{1, \|\grad f(X^k) \|_{{\Fsf}} \}$.  Here, the inequality \eqref{eq:inexact-1} is to guarantee the global convergence and the inequality \eqref{eq:inexact-2} leads to fast local convergence. Throughout the analysis of convergence, we assume that the stagnate conditions are never met. (In fact, a sufficient decrease for the original problem in each iteration can be guaranteed from the description of subspace refinement. Hence, the global convergence still holds.) 
 \subsection{Global convergence}
 Since the regularization term is used, the global convergence of our method can
 be obtained by assuming the boundedness on the constructed Hessian
 approximation. We first make the following assumptions. 
  \begin{assumption} \label{assum:global} Let $\{X^k\}$ be generated by
    Algorithm \ref{alg:struct-QN} without subspace refinement. We assume:
 \begin{itemize}
 \item[{(A1)}] The gradient $\nabla f $ is Lipschitz continuous on the convex hull of $\mathrm{St}(n,p)$, i.e., there exists $L_{f} > 0$ such that
 \end{itemize}
 \[ \|\nabla f(X) - \nabla f(Y)\|_{{\Fsf}} \leq L_{f}\| X - Y\|_{{\Fsf}}, \quad \forall~X,Y \in \conv({\mathrm{St}(n,p)}).\]
 \begin{itemize}
 \item[{(A2)}] There exists $\kH > 0$ such that $\|\Bcal^k\| \leq \kH$ for all $k \in \N$, where $\| \cdot \|$ is the operator norm introduced by the Euclidean inner product.
 \end{itemize}
 \end{assumption}

\begin{remark}
	By Assumption (A1),  $\nabla f(X)$ is uniformly bounded by some constant $\kappa_g$ on the compact set $\conv(\mathrm{St}(n,p))$, i.e., 
	$$ \|\nabla f(X)\|_{{\Fsf}} \leq \kg, \; X \in \conv(\mathrm{St}(n,p)). $$
	Assumption (A2) is often used in the traditional symmetric rank-1 method \cite{byrd1996analysis} which appears to be reasonable in practice. 
\end{remark}

Based on the similar proof in \cite{hu2018adaptive,wen2013adaptive}, we have the following theorem for global convergence.
\begin{theorem}
	Suppose that Assumptions (A1)-(A2) and the inexact conditions \eqref{eq:inexact-1} hold. Then, either 
	\be \grad f(X^t) = 0~for ~some ~t >0 \quad or \quad \lim_{k \rightarrow \infty} \|\grad f(X^k)\|_{{\Fsf}} = 0.\ee
\end{theorem} 
\begin{proof}
	For the quadratic regularization \eqref{eq:quad-dist-stiefel}, let us note that the Riemannian Hessian $\Hess m(X^k)$ can be guaranteed to be bounded from Assumption \ref{assum:global}. In fact, from \eqref{eq:hess-stiefel}, we have 
	\[ \|\Hess m_k(X^k) \| \leq \|\Bcal^k\| + \|X^k\| \|\nabla f(X^k)\|_{{\Fsf}} + \tau_k \leq \kH + \kg + \tau_k, \]
	where $\|X^k\| = 1$ because of its unitary property. 
	Hence, we can guarantee that the direction obtained from the modified CG
    method is a descent direction via similar techniques in \cite[Lemma 7]{hu2018adaptive}. Then the convergence of the iterates $\{X^k\}$ can be proved in a similar way by following the details in \cite{hu2018adaptive} for the quadratic regularization. As to the cubic regularization, we can refer \cite[Theorem 4.9]{wen2013adaptive} for a similar proof. 
\end{proof}
 
 \subsection{Local convergence}
	We now  focus on the local convergence with the inexact conditions
    \eqref{eq:inexact-1} and \eqref{eq:inexact-2}. We make some necessary
    assumptions below.
	\begin{assumption}
		Let $\{ X^k\}$ be the sequence generated by Algorithm
        \ref{alg:struct-QN} without subspace refinement. We assume
		\begin{itemize} 
			\item[(B1)] The sequence $\{X^k\}$ converges to $X_*$ with $\grad f(X_*) = 0$. 
			\item[(B2)] The Euclidean Hessian $\nabla^2 f$ is continuous on conv$(\mathrm{St}(n,p))$. 
			\item[(B3)] The Riemannian Hessian $\Hess f(X)$ is positive definite at $X_*$.  
			\item[(B4)] The Hessian approximation ${\Bcal^k}$ satisfies
			\be \label{eq:hessapprox} \frac{\|({\Bcal^k} - \nabla^2 f(X^k))[Z^k - X^k] \|_{{\Fsf}}}{\|Z^k - X^k \|_{{\Fsf}}} \rightarrow 0, ~ k \rightarrow \infty.  \ee
		\end{itemize} 
	\end{assumption}  
	
	Following the proof in \cite[Lemma 17]{hu2018adaptive}, we show that all iterations are eventually very successful (i.e., $r_k \geq \eta_2$, for all sufficiently large $k$) when Assumptions (B1)-(B4) and the inexact conditions \eqref{eq:inexact-1} and \eqref{eq:inexact-2} hold.
	\begin{lemma} \label{lemma:localconvergence}
		Let Assumptions (B1)-(B4) be satisfied. Then, all iterations
		are eventually very successful.
    \end{lemma}
    \begin{proof}
    	From the second-order Taylor expansion, we have 
    	\[ f(Z^k) - f(X^k) - m_k(Z^k) \leq \frac{1}{2} \Re \iprod{(\nabla^2 f(X^k_\delta) - {\Bcal^k})[Z^k - X^k])}{Z^k - X^k}, \]	
    	for some suitable $\delta_k \in [0,1]$ and $X^k_{\delta}:= X^k + \delta_k (Z^k - X^k)$. Since the Stifel manifold is compact, there exist some $\eta^k$ such that $Z^k = \mathrm{Exp}_{X^k}(\eta^k)$ where $\mathrm{Exp}_{X^k}$ is the exponential map from $T_{X^{k}} \mathrm{St}(n,p)$ to $\mathrm{St}(n,p)$. Following the proof in \cite[Appendix B]{boumal2016global} and Assumption (B1) ($Z^k$ can be sufficiently close to $X^k$ for large $k$), we have
    	\be \label{eq:sktk} \|Z^k - X^k - \eta^k\|_{{\Fsf}} \leq \kappa_1 \|\eta^k\|_{{\Fsf}}^2 \ee
    	with a positive constant $\kappa_1$ for all sufficiently large $k$. 
    	Moreover, since the Hessian $\Hess f(X_*)$ is positive definite and (B4)
        is satisfied, it holds for sufficiently large $k$:
    	\[ \begin{aligned}
    	\|\Hess &  m_k(X^k)[\eta^k] \|_{\Fsf} = \|\Hess m_k(X^k) [Z^k - X^k] \|_{\Fsf} +  O( \|\eta^k\|_{\Fsf}^2) \\
    	={} &  \| \Hess f(X^k)[Z^k - X^k] + (\Hess m_k(X^k) - \Hess f(X^k))[Z^k - X^k] \|_{\Fsf} + O( \|\eta^k\|_{\Fsf}^2)  \\
    	\geq{} & \lambda_{\min}(\Hess f(X^k)) \|Z^k - X^k\|_{\Fsf} + o(\|Z^k -X^k\|_{\Fsf}) + O( \|\eta^k\|_{\Fsf}^2) \\
    	\geq{} & \lambda_{\min}(\Hess f(X^k)) \|\eta^k\|_{\Fsf} + o( \|\eta^k\|_{\Fsf}),
    	\end{aligned}
    	\]
	    where $\lambda_{\min} (\Hess f(X^k))$ is the minimal spectrum of $\Hess f(X^k)$.  
    	From Assumption (B2)-(B3), \cite[Proposition 5.5.4]{opt-manifold-book} and the Taylor expansion of $m_k \circ \mathrm{Exp}_{X^k}$, we have
    	\bee \| \grad (m_k \circ \mathrm{Exp}_{X^k})(\eta^k) - \grad f(X^k) \|_{{\Fsf}} = \| \Hess f(X^k)[\eta^k] \|_{\Fsf} + o( \|\eta^k\|_{\Fsf}) 
    	\geq \frac{\kappa_2}{2} \|\eta^k\|_{{\Fsf}}  , \eee
    	where $\kappa_2 := \lambda_{\min}(\Hess f(X_*))$.
        By \cite[Lemma 7.4.9]{AbsilBakerGallivan2007}, we have 
    	\be \label{eq:gktk} \| \eta^k \|_{{\Fsf}} \leq \frac{2}{\kappa_2} (\|\grad f(X^k)\|_{{\Fsf}} + \tilde{c}\|\grad m_k(Z^k)\|_\Fsf \leq \frac{2(1 + \tilde{c} \theta^k)}{\kappa_2} \|\grad f(X^k)\|_{{\Fsf}}, \ee
    	where $\tilde{c} >0$ is a constant and the second inequality is from the inexact condition \eqref{eq:inexact-2}.
    	It follows from the continuity of $\nabla^2 f$, \eqref{eq:inexact-1}, \eqref{eq:sktk} and \eqref{eq:gktk} that
    	\beaa 1 - r_k \leq & \frac{1}{2c} \left( \frac{\|(\nabla^2 f(X^k) - {\Bcal^k})[Z^k - X^k]\|_{{\Fsf}}\|Z^k - X^k\|_{{\Fsf}} }{\|\grad f(X^k)\|_{{\Fsf}}^2} \right. \\
    	& \left. + \frac{\| \nabla^2 f(X^k_\delta) - \nabla^2 f(X^k) \| \|Z^k - X^k\|_{\Fsf}^2}{\|\grad f(X^k)\|_{{\Fsf}}^2} \right)
    	\rightarrow 0.\eeaa
    	Therefore, the iterations are eventually very successful. 
    \end{proof}	
	As a result, the q-superlinear convergence can also be guaranteed. 
	\begin{theorem} \label{thm:localconvergence}
		Suppose that Assumptions (B1)-(B4) and conditions \eqref{eq:inexact-1} and \eqref{eq:inexact-2} hold. Then the sequence $\{X^k\}$ converges q-superlinearly to $X_*$.
	\end{theorem}
	
	\begin{proof}
		We consider the cubic model here, while the local q-superlinear convergence of quadratic model can be showed by a similar fashion. Since the iterations are eventually very successful, we have $X^{k+1} = Z^k$ and $\tau_k$ converges to zero. From \eqref{eq:inexact-2}, we have
		\be \label{eq:inexact1} 
		\begin{aligned} \left\| \grad m_k(X^{k+1}) \right\|_{{\Fsf}} ={} & \left\| \proj_{X^{k+1}}\left( \nabla f(X^k) + {\Bcal^k}[{\Delta^k}] + \tau_k\|{\Delta^k}\|_{{\Fsf}}{\Delta^k}  \right) \right\|_{{\Fsf}} \\ 
		\leq{} & \theta^k \|\grad f(X^k)\|_{{\Fsf}}, \end{aligned}  \ee
		where ${\Delta^k} = Z^k - X^k$. 
		Hence,
		\be \label{eq:inexact2} \begin{aligned}  &\| \grad  f(X^{k+1})
          \|_{{\Fsf}} \\
          ={} & \left\| \proj_{X^{k+1}}\left( \nabla f(X^{k+1}) \right) \right\|_{{\Fsf}} \\
			 ={}&  \left\| \proj_{X^{k+1}}\left( \nabla f(X^{k}) + \nabla^2 f(X^k)[{\Delta^k}] + o(\|{\Delta^k}\|_{{\Fsf}}) \right) \right\|_{{\Fsf}} \\
			 ={}& \left\| \proj_{X^{k+1}}\left( \nabla f(X^{k}) + {\Bcal^k}[{\Delta^k}] + o(\|{\Delta^k}\|_{{\Fsf}})  + (\nabla^2 f(X^k) - {\Bcal^k})[{\Delta^k}] \right) \right \|_{{\Fsf}} \\
			 \leq{} & \theta^k \|\grad f(X^k)\|_{{\Fsf}} + o(\|{\Delta^k}\|_{{\Fsf}}). 		
		\end{aligned}  \ee
	   It follows from a similar argument to \eqref{eq:gktk} that there exists some constant $c_1$  
		\[  \|{\Delta^k}\|_{{\Fsf}}  \leq c_1 \|\grad f(X^k)\|_{{\Fsf}},  \]
		for sufficiently large $k$. 
		Therefore, from \eqref{eq:inexact2} and the definition of $\theta^k$, we have 
		\be \label{eq:superlinear-grad} \frac{\| \grad f(X^{k+1})\|_{{\Fsf}}}{\|\grad f(X^k)\|_{{\Fsf}}} \rightarrow 0.  \ee
		Combining \eqref{eq:superlinear-grad}, Assumption (B3) and \cite[Lemma 7.4.8]{opt-manifold-book}, it yields
		\[ \frac{\dist(X^{k+1}, X_*)}{\dist(X^k, X_*)} \rightarrow 0, \]
		where $\dist(X,Y)$ is the geodesic distance between $X$ and $Y$ which belong to $\mathrm{St}(n,p)$. 
	    This completes the proof. 
      \end{proof}

\section{Linear eigenvalue problem} \label{section:EIG}
In this section, we apply the aforementioned strategy to the following linear eigenvalue problem 
\be \label{prob:lineig} \min_{X \in \R^{n \times p}}\ f(X):=\half \tr(X^\top C X) \quad \st \quad  X^\top X = I_p,  \ee
where $C := A + B$. Here, $A, B \in \R^{n \times n}$ are symmetric matrices and we assume that the multiplication of $BX$ is much more expensive than that of $AX$. 
Motivated by \brev{the} quasi-Newton \brev{methods} and eliminating the linear term in subproblem \eqref{prob:struct-QN}, we investigate the multisecant \brev{conditions} in \cite{gratton2007multi}
\be \label{eq:multisecant1-1} \hat{B}^k X^k = B X^k, \quad  \hat{B}^k S^k = BS^k \ee
with $S^k = X^k -X^{k-1}$. 
By a brief induction, we have an equivalent form of \eqref{eq:multisecant1-1}
\be \label{eq:multisecant1-2} \hat{B}^k[X^{k-1}, \, X^k] = B[X^{k-1}, \, X^k]. \ee
Then, using the limited-memory Nystr\"om approximation, we obtain the approximated matrix $\hat{B}^k$ as
\be \label{newapproxB-1} \hat{B}^k = W^k((W^k)^\top O^k)^\dag W_k^\top  ,\ee
where 
\be \label{eq:sub-lineig} O^k = \mathbf{orth}(\mathrm{span}\{X^{k-1}, \, X^k\}), ~\mathrm{and}~ W^k = BO^k. \ee Here, $\mathbf{orth} (Z)$ is to find the orthogonal basis of the space spanned by $Z$. Therefore, an approximation $C^k$ to $C$ can be set as 
\be \label{eq:ck-lin} C^k = A + \hat{B}^k .\ee
Since the objective function is invariant under rotation, i.e., $f(XQ) = f(X)$
for orthogonal matrix $Q \in \R^{p\times p}$, we also wants to construct a
subproblem whose objective function inherits the same property. Therefore, we
use the distance function between $X^k$ and $X$ as $$d_p(X,X^k) = \| XX^\top - X^k (X^k)^\top \|_{\brev{\Fsf}}^2, $$ which has been considered in \cite{EdelmanAriasSmith1999,thogersen:074103,yang2007trust} for the electronic structure calculation.  
Since $X^k$ and $X$ are orthonormal matrices, we have
\be \label{eq:regularization-lineig}\begin{aligned}
d_p(X,X^k) & = \tr((XX^\top-X^k (X^k)^\top)(XX^\top-X^k (X^k)^\top) )  \\
& = 2p - 2\tr(X^\top X^k (X^k)^\top X),
\end{aligned} \ee
which implies that $d_p(X,X^k)$ is a quadratic function on $X$.
Consequently, the subproblem can be constructed as 
\be \label{prob:sub-linear} \min_{X \in \R^{n \times p}} \ m_k(X) \quad \st \quad X^\top X = I_p,  \ee
where $$m_k(X):= \frac{1}{2} \tr(X^\top C^k X) + \frac{\tau_k}{4} d_p(X,X^k).$$ 
From the equivalent expression of $d_p(X,X^k)$ in
\eqref{eq:regularization-lineig}, problem \eqref{prob:sub-linear} is actually a
linear eigenvalue problem
\[ \begin{aligned}
(A + \hat{B}^k - \tau_k X^k(X^k)^\top)X & = X \Lambda, \\
X^\top X & = I_p, \\
\end{aligned} \]
where $\Lambda$ is a diagonal matrix whose diagonal elements are the $p$
smallest eigenvalues of $A + \hat{B}^k - \tau_k X^k(X^k)^\top$. Due to the low
computational cost of $A + \hat{B}^k - \tau_k X^k(X^k)^\top$ compared to $A+B$,
the subproblem \eqref{prob:sub-linear} can be solved efficiently using existing
eigensolvers. As in Algorithm \ref{alg:struct-QN}, we first solve  subproblem
\eqref{prob:sub-linear} to obtain a trial point and compute the ratio \eqref{eq:ratio} between the actual reduction and predicted reduction based on this trial point. Then the iterate and regularization parameter are updated according to \eqref{eq:ratio} and \eqref{alg:tau-up}.  
 Note that it is not necessary to solve the subproblems highly accurately in
 practice. 
 
\subsection{Convergence} 
Although the convergence analysis in section \ref{sec:convergenceanalysis} is based on
the regularization terms \eqref{eq:quad-dist-stiefel} and
\eqref{eq:cubic-dist-stiefel},  similar results can be established
with the specified regularization term $\frac{\tau_k}{4}d_p(X,X^k)$ using the
sufficient descent condition \eqref{eq:inexact-1}. It follows from the construction of $C^k$ in \eqref{eq:ck-lin} that 
\[ \|C\|_2 \leq \|A\|_2 + \|B\|_2, \quad \|C^k\|_2 \leq \|A \|_2 + \|B\|_2 \]
for any given matrices $A$ and $B$. Hence, Assumptions (A1) and (A2) hold with
$L_f = \kH = \|A\|_2 + \|B\|_2$. We have the following theorem on the global convergence.

\begin{theorem} \label{thm:lin-globalconvergence}
	Suppose that the inexact
    condition \eqref{eq:inexact-1} holds. Then, for the Riemannian gradients, either
	\[ \left(I_n-X^t(X^t)^\top\right) (C X^t) = 0 \mathrm{~for~some~} t>0 \; \;
    \mathrm{or} \;\; \lim_{k \rightarrow \infty} \| \left( I_n-X^k(X^k)^\top \right) (C X^k) \|_\Fsf = 0. \] 
\end{theorem}
\begin{proof}
 It can be guaranteed that the distance $d_p(X,X^k)$ is very small for
 a large enough regularization parameter $\tau_k$ by a similar argument to \cite[Lemma 9]{hu2018adaptive}.
	Specifically, the reduction of the subproblem requires that
	\[ \iprod{Z^k}{C^k Z^k} + \frac{\tau_k}{4} \| Z^k(Z^k)^\top - X^k (X^k)^\top \|_\Fsf^2 - \iprod{X^k}{C^k X^k} \leq 0 .\]
	From the cyclic property of the trace operator, it holds that
	\[ \iprod{C^k}{Z^k(Z^k)^\top - X^k (X^k)^\top} + \frac{\tau_k}{4} \| Z^k(Z^k)^\top - X^k (X^k)^\top \|_\Fsf^2 \leq 0.\]
	Then
	\be \label{eq:bounddiff} \| Z^k(Z^k)^\top - X^k (X^k)^\top \|_\Fsf \leq \frac{4 \kH}{\tau_k}. \ee
	From the descent condition \eqref{eq:inexact-1} for the subproblem, there exists some positive constant $\nu$ such that
	\be \label{eq:reduction-model} m_k(Z^k) - m_k(X^k) \geq - \frac{\nu}{\tau_k} \| \grad f(X^k) \|_{\Fsf}^2.\ee
	Based on the properties of $C^k$ and $C$, we have
	\be \label{eq:approx-err-model} \begin{aligned}
		f(Z^k) & - f(X^k) - (m_k(Z^k) - m_k(X^k)) \\ &
		= \iprod{Z^k}{C Z^k} - \iprod{Z^k}{C^k Z^k} - \frac{\tau_k}{4} \| Z^k(Z^k)^\top - X^k (X^k)^\top \|_\Fsf^2 \\
		& \leq \iprod{C - C^k}{Z^k (Z^k)^\top} = \iprod{C - C^k}{\left( Z^k (Z^k)^\top - X^k (X^k)^\top\right)^2 } \\
		& \leq (L_f + \kH) \| Z^k (Z^k)^\top - X^k (X^k)^\top \|_\Fsf^2 \\
		& \leq \frac{16 \kH^2 (L_f + \kH)}{\tau_k^2},
	\end{aligned}
    \ee
    where the second equality is due to $CX^k = C^k X^k$, the unitary $Z^k$ and $X^k$, as well as 
    $$\iprod{C-C^k}{Z^k (Z^k)^\top X^k (X^k)^\top} = \iprod{C-C^k}{X^k (X^k)^\top Z^k (Z^k)^\top} =0. $$ 
	Combining \eqref{eq:reduction-model} and \eqref{eq:approx-err-model}, we have that 
	\[ 1- r_k = \frac{f(Z^k) - f(X^k) - (m_k(Z^k) - m_k(X^k))}{m_k(X^k) - m_k(Z^k)} \leq 1 - \eta_2 \]
	for sufficiently large $\tau_k$ as in \cite[Lemma 8]{hu2018adaptive}. 
	Since the subproblem is solved with some sufficient reduction, the reduction
    of the original objective $f$ holds for large $\tau_k$ (i.e., $r_k $ is
    close to 1). Then the convergence of the norm of the Riemannian gradient
    $(I_n
    - X^k(X^k)^\top) CX^k$ follows in a similar fashion as \cite[Theorem 11]{hu2018adaptive}.   
\end{proof}	
The ACE method in \cite{lin2017convergence} needs an estimation  $\beta$
explicitly such that $B - \beta I_n $ is negative definite. By
 considering an equivalent matrix $(A+\beta I_n) + (B-\beta I_n)$, the convergence of
 ACE to a global minimizer is given. On the other hand, our algorithmic
 framework uses an  adaptive strategy to choose $\tau_k$ to guarantee the
 convergence to a stationary point. By using similar proof techniques in
 \cite{lin2017convergence}, one may also establish the convergence to
 a global minimizer.

\section{{Electronic structure calculation}}  \label{section:ks:hf}
{}
\subsection{Formulation}
{We now introduce the KS and HF total minimization models and present their 
gradient and Hessian of the objective functions in these two models.} 
After some proper discretization, the wave functions of $p$ occupied states can be approximated {by a matrix
$X = [x_1, \ldots, x_p] \in \mathbb{C}^{n\times p}$ with $X^*X = I_p$},
where $n$ corresponds to the spatial degrees of freedom. {The charge density associated  with the occupied states is defined as}
\[{\rho(X) = \diag(XX^*).}\]
Unless otherwise specified, we use the abbreviation  {$\rho$ for  $\rho(X)$}  in the following.   
The total energy functional {is defined as}
\be \label{prob:ks-obj} 
	E_{\ks}(X) := \frac{1}{4}\tr(X^*LX) + \half \tr(X^*V_{\ion}X) + \half \sum_l \sum_i \zeta_l |x_i^* w_l|^2 + \frac{1}{4} \rho^\top L^\dag \rho + \half e^\top \epsilon_{\xc}(\rho),
\ee
where $L$ is a discretized Laplacian
operator, $V_{\ion}$ is the constant ionic pseudopotentials, $w_l$ represents a
discretized pseudopotential reference projection function, ~$\zeta_l$ is a
constant whose value is $\pm 1$, $e$ is a vector of all ones in $\mathbb{R}^n$, and $\epsilon_{\xc}$ is related to the exchange correlation energy. Therefore, the KS total energy minimization problem can be expressed as
\be \label{prob:ks}  \min_{X \in \C^{n\times p}} \  E_{\ks}(X) \quad \st \quad X^*X = I_p. \ee  

Let $ \mu_{\xc}(\rho) = \frac{\partial \epsilon_{\xc}(\rho)}{\partial \rho}$ and
denote  the Hamilton $H_{\ks}(X)$ by
\be \label{eq:ham-ks} {H_{\ks}(X)} := \half L + V_{\ion} + \sum_l \zeta_l w_lw_l^* + {\Diag}((\Re L^\dag) \rho) + {\Diag}(\mu_{\xc}(\rho)^* e).\ee
Then the Euclidean gradient of $E_{\ks}(X)$ is computed as
\be \label{eq:grad-ks} {\nabla} E_{\ks}(X) = {H_{\ks}(X)}X. \ee
{Under the assumption that $\epsilon_{\xc}(\rho(X))$ is twice differentiable
with respect to $\rho(X)$,} {Lemma 2.1 in \cite{wen2013adaptive} gives an explicit form of the Hessian of $E_{\ks}(X)$ as}
\be \label{eq:hess-ks} {\nabla^2} E_{\ks}(X)[U] = {H_{\ks} (X)} U + {\Rcal(X)}[U]{,} \ee
where {$U \in \C^{n \times p}$} and {${\Rcal(X)}[U]:= \Diag\left( \big(\Re L^\dag + \frac{\partial^2 \epsilon_{\xc}}{\partial \rho^2}e\big)(\bar{X} \odot U + X \odot \bar{U})e\right) X$}.

Compared with KSDFT, {the} HF theory can provide a more accurate model to electronic structure calculations by involving the Fock exchange operator. After discretization, the exchange-correlation operator $\Vcal(\cdot) : \C^{n \times n} \rightarrow \C^{n\times n}$ is usually a fourth-order tensor, {see equations (3.3) and (3.4) in \cite{le2005computational} for details}.  {Furthermore, it is easy to see from \cite{le2005computational} that $\Vcal(\cdot)$ satisfies the following properties: 
 (i)  For  any $D_1, D_2 \in \C^{n \times n}$, there holds 
$\iprod{\Vcal(D_1)}{D_2} = \iprod{\Vcal(D_2)}{D_1}, $
which further implies that 
\be \label{equ:V:property} 
\iprod{\Vcal(D_1 + D_2)}{D_1 + D_2} = \iprod{\Vcal(D_1)}{D_1} + 2\iprod{\Vcal(D_1)}{D_2} + \iprod{\Vcal(D_2)}{D_2}.
\ee
(ii)  If $D$ is Hermitian, $\Vcal(D)$ is also Hermitian.} 
{Besides, it should be emphasized}  that computing {$\Vcal(U)$} is always very expensive since it needs to  perform the multiplication between a $n \times n \times n \times n $ fourth-order tensor and a $n$-by-$n$ matrix.   
The corresponding  {Fock energy} is defined as 
\be \label{eq:fock-obj} E_{\f}(X):= \frac{1}{4}\iprod{\Vcal(XX^*) X}{X} = \frac{1}{4}\iprod{\Vcal(XX^*)}{XX^*}. \ee 
Then the HF total energy minimization problem can be formulated as 
\be \label{prob:hf} \min_{X \in \C^{n\times p}} \quad E_{\hf}(X):= E_{\ks}(X) + E_{\f}(X) \quad \st \quad X^*X = I_p. \ee
{We now can  explicitly compute the gradient and Hessian of $E_{\f}(X)$ by using the properties of $\Vcal(\cdot)$}. 
\begin{lemma} \label{lemma:grad-hess}
Given $U \in \C^{n\times p}$, the gradient and the Hessian along $U$ of $E_{\f}(X)$ are, respectively, 
 \begin{eqnarray}
  \label{eq:fock-grad} \nabla E_{\f}(X) & = & {\Vcal(XX^*)}X,  \\
 \label{eq:fock-hess} \nabla^2 E_{\f}(X)[U] & = & {\Vcal(XX^*)}U + \Vcal(XU^* + UX^*) X. 
\end{eqnarray}
\end{lemma}
\begin{proof}
{
We first compute  the value $E_{\f} (X + U)$. For simplicity, denote $D \coloneqq XU^* + U X^*$. Using the property \eqref{equ:V:property},  by some easy calculations, we have
\begin{align}
4 E_{\f} (X + U)  
 ={}&  \iprod{\Vcal\!\left((X+U)(X + U)^*\right)}{(X+U)(X + U)^*}  \nn  \\
={}& 4 E_{\f} (X) + 2 \iprod{\Vcal(XX^*)}{D + U U^*}  + \iprod{\Vcal(D + U U^*)}{D + U U^*}  \nonumber \\ 
={}& 4 E_{\f} (X) + 2 \iprod{\Vcal(XX^*)}{D}  + 2 \iprod{\Vcal(XX^*)}{U U^*}  + \iprod{\Vcal(D)}{D} + \mbox{h.o.t.},  \nn  
\end{align} 
where h.o.t. denotes the higher-order terms. Noting that $\Vcal(XX^*)$ and $\Vcal(D)$ are both Hermitian, we have from the above assertions that 
\be\label{equ:2ndTaylor:Ef}
 E_{\f} (X + U)   = E_{\f} (X) +   \Re \iprod{\Vcal(XX^*) X}{U}  
 + \frac12  \Re \iprod{\Vcal(XX^*) U + \Vcal(D) X}{U} + \mbox{h.o.t.}.
\ee
 Finally, it follows from expansion (1.2) in \cite{wen2013adaptive} that the second-order Taylor expression in $X$ can be expressed as 
\[
E_{\f} (X + U) = E_{\f}(X) + \Re\langle \nabla E_{\f}(X), U \rangle + \frac12 \Re \iprod{\nabla^2 E_{\f}(X)[U]}{U} + \mbox{h.o.t.},
\]
which with \eqref{equ:2ndTaylor:Ef} implies \eqref{eq:fock-grad} and \eqref{eq:fock-hess}. The proof is completed. }
\end{proof} 

{Let   ${H_{\hf}(X) := H_{\ks}(X) + \Vcal(XX^*)}$ be the HF Hamilton}.  {Recalling that $E_{\hf}(X) =  E_{\ks}(X) + E_{\f}(X)$, we have from  \eqref{eq:grad-ks} and \eqref{eq:fock-grad} that}
\be\label{eq:grad-hf}
 \nabla E_{\hf}(X)  = {H_{\ks}(X)X + \Vcal(XX^*)X = H_{\hf}(X)X}
 \ee
{and have from \eqref{eq:hess-ks} and \eqref{eq:fock-hess} that}
\be \label{eq:hess-hf}
 {\nabla^2 E_{\hf}(X) [U]  = H_{\hf}(X) U + \Rcal(X)[U] +  \Vcal(XU^* + UX^*) X.}
\ee
\subsection{Review of Algorithms for the KSDFT and HF Models}
{We next briefly introduce the widely used methods for solving the KSDFT and HF models.} 
For the KSDFT model \eqref{prob:ks}, the most popular {method} is the SCF method \cite{le2005computational}. 
At the {$k$-th iteration}, SCF first fixes {$H_{\ks}(X)$} to be $H(X^k)$ and  {then updates  $X^{k+1}$ via solving the linear eigenvalue problem} 
\be \label{alg:scf}
{
 {X^{k+1} \coloneqq} \argmin_{X \in \C^{n \times p}}\  \frac12 \langle X, H(X^k) X\rangle \quad \st \quad X^* X = I_p.
}
\ee
Because {the complexity of the HF model \eqref{prob:hf} is much higher than that of the KSDFT model}, using SCF method directly may not obtain desired results. 
 Since computing $\Vcal\left(X^k(X^k)^*\right) U$ with some matrix $U$ of proper dimension is still very expensive, we investigate the limited-memory Nystr\"om approximation {$\hat{\Vcal}\left(X^k(X^k)^*\right)$ to approximate $\Vcal\left(X^k(X^k)^*\right)$} to reduce the computational cost, i.e.,
\be \label{eq:nys} {\hat{\Vcal}\left(X^k(X^k)^*\right)} := Z(Z^*\Omega)^\dag Z^*, \ee
where $Z= {\Vcal\left(X^k(X^k)^*\right)\Omega}$ and $\Omega$ is {any orthogonal matrix whose columns form an orthogonal basis of the} subspace such as  
\[ \sspan \{X^k\}, ~ \sspan \{X^{k-1}, X^k\} ~\mathrm{or}~  \sspan \{X^{k-1}, X^k, {\Vcal\left(X^k(X^k)^* \right)}X^k\}. \] 
We should note that a similar idea called adaptive compression method was proposed in {\cite{lin2016adaptively}}, which only considers to compress the operator $\Vcal(X^k(X^k)^*)$ on the subspace $\sspan \{ X^k\}$.  
Then a new subproblem is constructed as
\be 
\label{prob:2scf-inner} \min_{X \in \C^{n\times p}} \quad E_{\ks}(X) + \frac{1}{4}\iprod{{\hat{\Vcal}\left(X^k(X^k)^*\right)X}}{X} \quad \st \quad X^*X = I_p.
\ee
Here, the exact form of the easier parts $E_{\ks}$ is preserved while its second-order approximation is used in the construction of subproblem \eqref{prob:struct-QN}.
As in the subproblem \eqref{prob:struct-QN}, we can utilize the Riemannian gradient method or the modified CG method based on the following linear equation
\[ \proj_{X^k} {\left(\nabla^2 E_{\ks}(X^k)[\xi] + \frac{1}{2} \hat{\Vcal}(X^k(X^k)^*)\xi - \xi \sym((X^k)^*\nabla f(X^k))\right)} = - \grad E_{\hf}(X^k)  \]
to solve \eqref{prob:2scf-inner} inexactly.
Since \eqref{prob:2scf-inner} is a KS-like problem, we can also use the SCF method. Here, we present the detailed algorithm in Algorithm \ref{alg:ace}. \\
\begin{algorithm2e}[H]\label{alg:ace}
	\caption{Iterative method for \eqref{prob:hf} using Nystr\"om approximation}
	Input initial guess $X^0 \in \C^{n\times p}$ with $(X^0)^*X^0 = I_p$. Set $k = 0$. \\
	\While{Stopping condtions not met}{
		{Compute the limited-memory Nystr\"om approximation $\hat{\Vcal}\left(X^k(X^k)^*\right)$}. \\	
		Construct the subproblem \eqref{prob:2scf-inner} and solve it inexactly	
	    via the Riemannian gradient method or the modified CG method or {the SCF method} to obtain $X^{k+1}$. \\
		Set $k \gets k+1$.
	}
\end{algorithm2e}
We note that Algorithm \ref{alg:ace} is similar to the two-level nested SCF
method with the ACE formulation \cite{lin2016adaptively} when the subspace in \eqref{eq:nys} and inner solver for \eqref{prob:2scf-inner} are chosen as $\sspan\{ X^k\}$ and SCF, respectively. 

\subsection{{Construction of the structured approximation $\Bcal^k$}}
{Note that the Hessian of the KSDFT or HF total energy minimization takes the natural structure \eqref{equ:f:hessian:structure}, we next give the specific choices of  $\Hc(X^k)$ and $\He(X^k)$, which are key to formulate the the structured approximation $\Bcal^k$.}

For the KS problem \eqref{prob:ks}, we have its exact Hessian in
\eqref{eq:hess-ks}. Since the computational cost of the parts $\half L +
\sum_{l} \zeta_l w_lw_l^*$ are much cheaper than the remaining parts in $\nabla^2
E_{\ks}$, we can choose
 \be \label{eq:qn-ks1} \Hc(X^k) = \half L + \sum_{l} \zeta_l w_lw_l^*, \quad \He(X^k) = \nabla^2 E_{\ks}(X^k) -  \Hc(X^k). \ee

 {The exact Hessian of $E_{\hf}(X)$ in} \eqref{prob:hf} can be separated {naturally into} two parts, i.e., $\nabla^2 E_{\ks}(X) + \nabla^2 E_{\f}(X)$.
{Usually} the hybrid exchange operator {$\Vcal(XX^*)$} can take more than $95\%$ of the
overall time of the multiplication of {$H_{\hf}(X)[U]$} in many real applications \cite{lin2017convergence}. 
{Recalling  \eqref{eq:hess-ks}, \eqref{eq:fock-hess}  and \eqref{eq:hess-hf}, we know that}  the computational cost of {$\nabla^2 E_{\f} (X)$ is much higher than that of $\nabla^2 E_{\ks}(X)$}. Hence, {we obtain the decomposition as}
 \be \label{struct-B-app} \Hc(X^k) = \nabla^2 E_{\ks}(X^k), \quad \He(X^k) = \nabla^2 E_{\f}(X^k). \ee
 Moreover, we can split the Hessian of $\nabla^2 E_{\ks}(X^k)$ as {done} in \eqref{eq:qn-ks1} {and obtain an alternative decomposition as}
 \be \label{eq:qn-hf} \Hc(X^k) = {H_{\ks}(X^k)}, \quad \He(X^k) = {\nabla^2 E_{\f}(X^k) +  (\nabla^2 E_{\ks}(X^k) - {\Hc(X^k)})}. \ee
 
 {Finally, we emphasize} that the limited-memory Nystr\"om approximation \eqref{eq:nys} can serve as {a good} initial approximation for the part $\nabla^2 E_{\f}(X^k)$.

 \subsection{Subspace construction for the KSDFT model}
 As presented in Algorithm \ref{alg:struct-QN}, the subspace method  plays an
 important role when the modified CG method does not perform well. The first-order optimality conditions for \eqref{prob:ks} and \eqref{prob:hf} are
 \[ H(X)X = X\Lambda, \quad X^*X = I_p, \]
 where $X \in \C^{n\times p}$, $\Lambda$ is a diagonal matrix and $H$ represents
 $H_{\ks}$ for \eqref{prob:ks} and $H_{\hf}$ for \eqref{prob:hf}. Then, problems
 \eqref{prob:ks} and \eqref{prob:hf} are actually a nonlinear eigenvalue problem
 which aims to find the $p$ smallest eigenvalues of $H$. We should point out that in principle $X$ consists of the eigenvectors of $H(X)$ but not necessary the eigenvectors corresponding to the $p$ smallest eigenvalues. 
 Since the optima $X$ is still the eigenvectors of $H(X)$, we can construct some
 subspace which contains these possible {wanted} eigenvectors. Specifically, at
 current iterate, we first compute the first $\gamma p$ smallest eigenvalues and
 their corresponding eigenvectors of $H(X^k)$, denoted by $\Gamma^k$, then construct the subspace as
 \be \label{eq:subspace}  {\sspan\{ X^{k-1},X^{k}, \grad f(X^k), \Gamma^{k}\},} \ee
 with some small integer $\gamma$. 
 With this subspace construction, Algorithm \ref{alg:struct-QN} will more likely
 escape a stagnated point.

\section{Numerical experiments} \label{section:NUM}
In this section, we present some experiment results to illustrate the efficiency
of the limited-memory Nystr\"om approximation and our Algorithm
\ref{alg:struct-QN}. All codes were run in a workstation with Intel Xenon E5-2680 v4 processors at 2.40GHz and 256GB memory running CentOS 7.3.1611 and MATLAB R2017b. 
\subsection{Linear eigenvalue problem} 
We first construct $A$ and $B$ by using the following MATLAB commands:
\[ A = \mathrm{randn}(n,n); \, A = (A + A^\top)/2; \]
\[ B = 0.01 \mathrm{rand}(n,n);\, B = (B + B^\top)/2; \, B = B - T; \, B = -B,  \]
where $\mathrm{randn}$ and $\mathrm{rand}$ are the built-in functions in MATLAB,
$T= \lambda_{\min}(B)I_n$ and $\lambda_{\min}(B)$ is the smallest eigenvalue
of $B$.  Then $B$ is negative definite and $A$ is symmetric. In our
implementation, we compute the multiplication $BX$ using
 $\frac{1}{19} \sum_{i=1}^{19} BX$
such that $BX$ consumes about $95\%$ of the whole computational time.  In the
second example, we set $A$ to be a sparse matrix as 
\[ A = \mathrm{gallery}(\textrm{`wathen'},5s,5s)  \]
with parameter $s$ and $B$ is the same as the first example except that 
$BX$ is computed directly. 
Since $A$ is sufficiently sparse, its computational cost $AX$ is much smaller
than that of $BX$. 
We use the following stopping criterion
\be \label{stop:err} \mathrm{err}:=\max_{i = 1, \ldots, p} \left\{ \frac{\|(A+B)x_i - \mu_i x_i\|_2}{\max(1, |\mu_i|)}  \right\} \leq 10^{-10}, \ee 
where $x_i$ is the $i$-th column of the current iterate $X^k$ and $\mu_i$ is the
 corresponding approximated eigenvalue.

The numerical results of the first and second examples are summarized in  Tables
\ref{tab:rand} and \ref{tab:sparse}, respectively. In these tables, EIGS is the
built-in function ``eigs'' in MATLAB. LOBPCG is the locally optimal block
preconditioned conjugate gradient method \cite{knyazev2001toward}. ASQN is the
algorithm described in section \ref{section:EIG}. {The} difference between ACE and
ASQN is that we take $O^k$ as $\mathbf{orth}(\sspan\{X^k\})$ but not
$\mathbf{orth}(\sspan\{X^{k-1}, X^k\})$. Since  a good initial guess $X^k$ is
known at the $(k+1)$-th iteration, LOBPCG is utilized to solve the corresponding
linear eigenvalue subproblem \eqref{prob:sub-linear}. Note that $BX^{k-1}$ and
$BX^{k}$ are available from the computation of the residual, we then adopt the
orthogonalization technique in \cite{liu2013limited} to compute $O^k$ and $W^k$
in \eqref{eq:sub-lineig} without extra multiplication $BO^k$. The labels 
``AV'' and ``BV'' denote the total number of matrix-vector multiplications (MV), counting
each operation $AX, BX \in \R^{n\times p}$ as $p$ MVs.
The columns  ``err'' and ``time''  are the maximal relative error of all $p$
eigenvectors defined in \eqref{stop:err}, and the wall-clock time in seconds of each
algorithm, {respectively}. The maximal number of iterations for ASQN and ACE is
set to 200.
 
As shown in Table \ref{tab:rand}, with fixed $p = 10$ and different $n = 5000,
6000, 8000$ and $10000$, we see that ASQN performs better than EIGS, LOBPCG and
ACE in terms of both accuracy and time. ACE spends a relative long time to reach
a solution with a similar accuracy. For the case $n = 5000$, ASQN can still give
a accurate solution with less time than EIGS and LOBPCG, but ACE usually
takes a long time to get a solution of high accuracy. Similar
conclusions can also be seen from Table \ref{tab:sparse}. In which, ACE and LOBPCG do not reach the given accuracy in the cases $n=11041$ and $p=30,40,50,60$.
From the calls of $AV$ and $BV$, we see that the limited-memory Nystr\"om method reduces the number of calls on the expensive part by doing more evaluations on the cheap part. 

\begin{table} \label{tab:rand} \caption{Numerical results on random matrices}
	\footnotesize
	\centering
	\setlength{\tabcolsep}{1.5pt}
	\begin{tabular}{|c|ccc|ccc|}
		\hline
		
		& 	 AV/BV & 	 err & 	 time& 	 AV/BV & 	 err & 	 time \\ \hline
		\multicolumn{7}{|c|}{$ p = 10$} \\ \hline
		$n$ & 	 \multicolumn{3}{c|}{5000}& 	 \multicolumn{3}{c|}{6000} \\ \hline
		
		EIGS& 459/459 & 	 8.0e-11 & 	 45.1& 730/730 & 	 6.9e-11 & 	 94.3\\ \hline 
		LOBPCG& 	 1717/1717 & 	 9.9e-11 & 	 128.9& 	 2105/2105 & 	 9.8e-11 & 	 249.9\\ \hline 
		ASQN& 	 2323/150 & 	 9.2e-11 & 	 13.3& 	 2798/160 & 	 9.5e-11 & 	 22.8\\ \hline 
		ACE& 	 4056/460 & 	 9.7e-11 & 	 30.8& 	 4721/460 & 	 9.4e-11 & 	 47.4\\ \hline   
		
        $n$ & 	 \multicolumn{3}{c|}{8000}& 	 \multicolumn{3}{c|}{10000}\\ \hline
        EIGS& 538/538 & 	 8.7e-11 & 	 131.9& 981/981 & 	 8.8e-11 & 	 327.3\\ \hline 
        LOBPCG& 	 1996/1996 & 	 9.9e-11 & 	 336.7& 	 2440/2440 & 	 9.7e-11 & 	 763.8\\ \hline 
        ASQN& 	 2706/150 & 	 8.9e-11 & 	 29.8& 	 2920/150 & 	 9.7e-11 & 	 50.2\\ \hline 
        ACE& 	 4537/450 & 	 9.8e-11 & 	 66.3& 	 4554/400 & 	 9.6e-11 & 	 99.4\\ \hline

		\multicolumn{7}{|c|}{$n = 5000$} \\ \hline
		$p$ & 	 \multicolumn{3}{c|}{10}& 	 \multicolumn{3}{c|}{20} \\ \hline 	 
		EIGS& 459/459 & 	 8.0e-11 & 	 45.1& 638/638 & 	 3.2e-11 & 	 62.7\\ \hline 
		LOBPCG& 	 1717/1717 & 	 9.9e-11 & 	 128.9& 	 2914/2914 & 	 9.8e-11 & 	 130.3\\ \hline 
		ASQN& 	 2323/150 & 	 9.2e-11 & 	 13.3& 	 3809/260 & 	 9.2e-11 & 	 8.9\\ \hline 
		ACE& 	 4056/460 & 	 9.7e-11 & 	 30.8& 	 5902/680 & 	 9.5e-11 & 	 16.5\\ \hline  
				
		$p$ & \multicolumn{3}{c|}{30}& 	 \multicolumn{3}{c|}{50}\\ \hline 
		EIGS& 660/660 & 	 3.0e-11 & 	 62.8& 879/879 & 	 1.6e-12 & 	 83.6\\ \hline 
		LOBPCG& 	 4458/4458 & 	 1.0e-10 & 	 217.6& 	 5766/5766 & 	 9.5e-11 & 	 186.7\\ \hline 
		ASQN& 	 5315/420 & 	 9.8e-11 & 	 11.4& 	 7879/650 & 	 9.8e-11 & 	 17.8\\ \hline 
		ACE& 	 9701/1530 & 	 9.4e-11 & 	 23.0& 	 21664/4450 & 	 1.0e-10 & 	 50.9\\ \hline

	\end{tabular}
	
\end{table}

\begin{table} \label{tab:sparse} \caption{Numerical results on sparse matrices}
	\footnotesize
	\centering
	\setlength{\tabcolsep}{1.5pt}
	\begin{tabular}{|c|ccc|ccc|}
		\hline
		
		& 	 AV/BV & 	 err & 	 time& 	 AV/BV & 	 err & 	 time \\ \hline
		\multicolumn{7}{|c|}{$ p = 10$} \\ \hline
		$s$ & 	\multicolumn{3}{c|}{7}& 	 \multicolumn{3}{c|}{8} \\ \hline
		
		EIGS& 1589/1589 & 	 8.9e-11 & 	 10.8& 1097/1097 & 	 6.1e-11 & 	 13.4\\ \hline 
		LOBPCG& 	 3346/3346 & 	 9.8e-11 & 	 24.6& 	 4685/4685 & 	 4.6e-10 & 	 48.6\\ \hline 
		ASQN& 	 5387/180 & 	 9.6e-11 & 	 7.1& 	 4861/150 & 	 9.9e-11 & 	 5.9\\ \hline 
		ACE& 	 14361/1600 & 	 9.6e-11 & 	 21.7& 	 8810/600 & 	 9.6e-11 & 	 12.7\\ \hline 
		
		$s$ & 	 \multicolumn{3}{c|}{9}& 	 \multicolumn{3}{c|}{10}\\ \hline
		EIGS& 1326/1326 & 	 9.3e-11 & 	 21.2& 1890/1890 & 	 6.8e-11 & 	 44.4\\ \hline 
		LOBPCG& 	 4306/4306 & 	 1.7e-07 & 	 66.9& 	 3895/3895 & 	 9.9e-11 & 	 91.9\\ \hline 
		ASQN& 	 5303/190 & 	 8.5e-11 & 	 7.7& 	 6198/200 & 	 8.9e-11 & 	 10.1\\ \hline 
		ACE& 	 16253/1850 & 	 9.9e-11 & 	 34.6& 	 10760/820 & 	 9.0e-11 & 	 22.2\\ \hline 
						
		$s$ & 	 \multicolumn{3}{c|}{11}& 	 \multicolumn{3}{c|}{12} \\ \hline 	 
		EIGS& 1882/1882 & 	 1.5e-07 & 	 58.9& 1463/1463 & 	 9.6e-11 & 	 65.4\\ \hline 
		LOBPCG& 	 4282/4282 & 	 9.5e-11 & 	 136.0& 	 4089/4089 & 	 9.9e-11 & 	 190.6\\ \hline 
		ASQN& 	 8327/240 & 	 9.6e-11 & 	 16.7& 	 6910/220 & 	 9.3e-11 & 	 17.5\\ \hline 
		ACE& 	 15323/1060 & 	 9.7e-11 & 	 38.9& 	 17907/2010 & 	 1.7e-08 & 	 65.5\\ \hline 
				
		\multicolumn{7}{|c|}{$ s=12$} \\ \hline
		$p$ & 	\multicolumn{3}{c|}{10}& 	 \multicolumn{3}{c|}{20} \\ \hline
		
		EIGS& 1463/1463 & 	 9.6e-11 & 	 65.4& 1148/1148 & 	 5.8e-11 & 	 50.2\\ \hline 
		LOBPCG& 	 4089/4089 & 	 9.9e-11 & 	 190.6& 	 5530/5530 & 	 9.8e-11 & 	 86.4\\ \hline 
		ASQN& 	 6910/220 & 	 9.3e-11 & 	 17.5& 	 9749/340 & 	 9.5e-11 & 	 16.3\\ \hline 
		ACE& 	 17907/2010 & 	 1.7e-08 & 	 65.5& 	 14108/960 & 	 9.8e-11 & 	 23.4\\ \hline  
		
		$p$ & 	 \multicolumn{3}{c|}{30}& 	 \multicolumn{3}{c|}{40}\\ \hline
		EIGS& 1784/1784 & 	 8.1e-11 & 	 74.8& 1836/1836 & 	 4.8e-11 & 	 69.1\\ \hline 
		LOBPCG& 	 9076/9076 & 	 5.3e-09 & 	 173.3& 	 12192/12192 & 	 4.6e-10 & 	 207.2\\ \hline 
		ASQN& 	 17056/870 & 	 9.6e-11 & 	 41.5& 	 19967/960 & 	 9.9e-11 & 	 39.9\\ \hline 
		ACE& 	 37162/6030 & 	 9.1e-09 & 	 78.4& 	 48098/8040 & 	 4.6e-07 & 	 105.4\\ \hline 
				
		$p$ & 	 \multicolumn{3}{c|}{50}& 	 \multicolumn{3}{c|}{60} \\ \hline 	 
		EIGS& 1743/1743 & 	 7.3e-11 & 	 69.1& 2122/2122 & 	 1.6e-11 & 	 86.7\\ \hline 
		LOBPCG& 	 12288/12288 & 	 1.4e-09 & 	 168.4& 	 15716/15716 & 	 1.1e-08 & 	 199.5\\ \hline 
		ASQN& 	 21330/1300 & 	 9.3e-11 & 	 53.6& 	 26343/1620 & 	 9.7e-11 & 	 71.8\\ \hline 
		ACE& 	 49165/10050 & 	 2.9e-06 & 	 110.1& 	 62668/12060 & 	 2.3e-08 & 	 134.0\\ \hline 		
		
	\end{tabular}
	
\end{table}

\subsection{Kohn-Sham total energy minimization} \label{sec:ks}
We now test the  electron structure calculation
models in subsections \ref{sec:ks} and \ref{sec:hf} using the new version of
the KSSOLV package \cite{YangMezaLeeWang2009}. 
One of the main differences is that the new
version uses the more recently developed optimized norm-conserving Vanderbilt pseudopotentials (ONCV) \cite{hamann2013optimized}, which are compatible to those used in other community software packages such as Quantum ESPRESSO.  
The problem information is listed in Table \ref{tab:prob}. For fair
comparisons, we stop all algorithms when the Frobenius norm of the Riemannian gradient is less than $10^{-6}$ or the maximal number of iterations {is} reached. In the following tables, the column ``solver'' denotes which specified solver is used. The columns ``fval'', ``nrmG'', ``time'' are the final objective function value, the final Frobenius norm of the Riemannian gradient and the wall-clock time in seconds of each algorithm, {respectively}. 

In this test, 
we compare structured quasi-Newton method with {the} SCF in KSSOLV \cite{YangMezaLeeWang2009},  {the} Riemannian L-BFGS method (RQN) in Manopt \cite{manopt}, {the} Riemannian gradient method with BB step size (GBB) and  {the} adaptive regularized Newton method (ARNT) \cite{hu2018adaptive}.
The default parameters therein are used.  
{Our} Algorithm \ref{alg:struct-QN} with the approximation with \eqref{eq:qn-ks1} is denoted by ASQN.  {The parameters setting of ASQN is same to that} of ARNT \cite{hu2018adaptive}. 

{For each algorithm, we first} use GBB to generate a good starting point with stopping criterion 
$\|\grad f(X^k)\|_{{\Fsf}} \leq 10^{-1}$ and a maximum of $2000$ iterations. The
maximal {numbers} of iterations for SCF, GBB, ARNT, ASQN and RQN are set as
1000, 10000, 500, 500, 500 and 1000, respectively. The numerical results are
reported in Tables \ref{tab:ks} and \ref{tab:kshard}. The column ``its'' represents the total number
of iterations in SCF, GBB and RQN, while the two numbers in ARNT, ASQN are the
total 
number of outer iterations and the average numbers of inner iterations.

\begin{table}[!htbp]
	\centering
	\footnotesize
	\caption{Problem information.} \label{tab:prob} 
	\setlength{\tabcolsep}{1.5pt}
	\begin{tabular}{|c|c|c|c|}
		\hline 
		name  & 	$(n_1,n_2,n_3)$ & 	 $n$ & 	 $p$  \\ \hline 
		alanine & (91,68,61) & 35829 & 18 \\ \hline
		c12h26  &  (136,68,28) & 16099 & 37 \\ \hline
		ctube661 & (162,162,21) & 35475 & 48 \\ \hline
		glutamine & (64,55,74) & 16517 & 29 \\ \hline
		graphene16 & (91,91,23) & 12015 & 37 \\ \hline
		graphene30 & (181,181,23) & 48019 & 67 \\ \hline
		pentacene & (80,55,160) & 44791 & 51 \\ \hline
		gaas & (49,49,49) & 7153 & 36 \\ \hline
		si40 & (129,129,129) & 140089 & 80 \\ \hline
		si64 & (93,93,93) & 51627 & 128 \\ \hline
		al & (91,91,91) & 47833 & 12 \\ \hline
		ptnio & (89,48,42) & 11471 & 43 \\ \hline
		c & (46,46,46) & 6031 & 2 \\ \hline
	\end{tabular} 
\end{table}

\begin{table}[!htbp]
	\centering
	\footnotesize
	\caption{Numerical results on KS total energy minimization.}
	\label{tab:ks} 
	\setlength{\tabcolsep}{1.5pt}
	\begin{tabular}{|c|c|c|c|c||c|c|c|c|}
		\hline 
		solver  & 	 fval & 	 nrmG & 	 its  & 	 time  & 	 fval & 	 nrmG & 	 its  & 	 time \\ \hline 
		& 	 \multicolumn{4}{c||}{alanine} & \multicolumn{4}{c|}{c12h26} \\ \hline 		
		SCF            & 	 -6.27084e+1 & 	 6.3e-7 & 	 11 & 	 64.0& 	 -8.23006e+1 & 	 6.5e-7 & 	 10 & 	 61.1\\ \hline 
		GBB            & 	 -6.27084e+1 & 	 8.2e-7 & 	 92 & 	 71.3& 	 -8.23006e+1 & 	 9.5e-7 & 	 89 & 	 65.8\\ \hline 
		ARNT    & 	 -6.27084e+1 & 	 3.8e-7 & 	  3(13.3) & 	 63.0& 	 -8.23006e+1 & 	 7.5e-7 & 	  3(15.3) & 	 60.9\\ \hline 
		ASQN    & 	 -6.27084e+1 & 	 9.3e-7 & 	 13(11.8) & 	 81.9& 	 -8.23006e+1 & 	 9.3e-7 & 	 10(13.3) & 	 67.8\\ \hline 
		RQN    & 	 -6.27084e+1 & 	 1.5e-6 & 	 34 & 	 114.9& 	 -8.23006e+1 & 	 1.7e-6 & 	 45 & 	 120.0\\ \hline 
		& 	 \multicolumn{4}{c||}{ctube661} & \multicolumn{4}{c|}{glutamine} \\ \hline 
		SCF            & 	 -1.35378e+2 & 	 5.7e-7 & 	 11 & 	 200.4& 	 -9.90525e+1 & 	 4.9e-7 & 	 10 & 	 49.5\\ \hline 
		GBB            & 	 -1.35378e+2 & 	 6.3e-7 & 	 102 & 	 199.7& 	 -9.90525e+1 & 	 4.9e-7 & 	 63 & 	 44.0\\ \hline 
		ARNT    & 	 -1.35378e+2 & 	 3.2e-7 & 	  3(18.3) & 	 168.3& 	 -9.90525e+1 & 	 3.6e-7 & 	  3(12.0) & 	 42.6\\ \hline 
		ASQN    & 	 -1.35378e+2 & 	 7.6e-7 & 	 11(12.8) & 	 201.7& 	 -9.90525e+1 & 	 5.3e-7 & 	 12(9.8) & 	 50.7\\ \hline 
		RQN    & 	 -1.35378e+2 & 	 3.4e-6 & 	 40 & 	 308.8& 	 -9.90525e+1 & 	 1.8e-6 & 	 26 & 	 72.8\\ \hline 
		& 	 \multicolumn{4}{c||}{graphene16} & \multicolumn{4}{c|}{graphene30} \\ \hline 
		SCF            & 	 -9.57196e+1 &   8.7e-4 &   1000 &   3438.4 & 	 -1.76663e+2 &   3.5e-4 &   1000 &   31897.6  \\ \hline 
		GBB            & 	 -9.57220e+1 & 	 9.4e-7 & 	 434 & 	 185.1& 	 -1.76663e+2 & 	 9.0e-7 & 	 904 & 	 3383.9\\ \hline 
		ARNT    & 	 -9.57220e+1 & 	 1.8e-7 & 	  4(37.2) & 	 164.1& 	 -1.76663e+2 & 	 4.2e-7 & 	  5(74.2) & 	 2386.1\\ \hline 
		ASQN    & 	 -9.57220e+1 & 	 8.8e-7 & 	 23(24.1) & 	 221.2& 	 -1.76663e+2 & 	 7.2e-7 & 	 74(31.1) & 	 4388.1\\ \hline 
		RQN    & 	 -9.57220e+1 & 	 1.6e-6 & 	 213 & 	 287.8& 	 -1.76663e+2 & 	 3.3e-5 & 	 373 & 	 4296.7\\ \hline 
		& 	 \multicolumn{4}{c||}{pentacene} & \multicolumn{4}{c|}{gaas} \\ \hline 
		SCF            & 	 -1.30846e+2 & 	 8.5e-7 & 	 12 & 	 279.8& 	 -2.86349e+2 & 	 5.8e-7 & 	 15 & 	 41.1\\ \hline 
		GBB            & 	 -1.30846e+2 & 	 9.6e-7 & 	 101 & 	 236.1& 	 -2.86349e+2 & 	 7.5e-7 & 	 296 & 	 77.7\\ \hline 
		ARNT    & 	 -1.30846e+2 & 	 2.1e-7 & 	  3(14.0) & 	 213.6& 	  -2.86349e+2 & 	 7.4e-7 & 	  3(46.3) & 	 59.9\\ \hline 
		ASQN    & 	 -1.30846e+2 & 	 9.0e-7 & 	 23(14.5) & 	 423.0& 	 -2.86349e+2 & 	 6.0e-7 & 	 35(24.8) & 	 127.2\\ \hline 
		RQN    & 	 -1.30846e+2 & 	 2.1e-6 & 	 34 & 	 437.9& 	 -2.86349e+2 & 	 1.5e-6 & 	 111 & 	 116.0\\ \hline 
		& 	 \multicolumn{4}{c||}{si40} & \multicolumn{4}{c|}{si64} \\ \hline 
		SCF            & 	 -1.57698e+2 & 	 7.5e-7 & 	 19 & 	 3587.4& 	 -2.53730e+2 & 	 3.4e-7 & 	 10 & 	 1100.0\\ \hline 
		GBB            & 	 -1.57698e+2 & 	 8.7e-7 & 	 289 & 	 3657.2& 	 -2.53730e+2 & 	 7.3e-7 & 	 249 & 	 1534.2\\ \hline 
		ARNT    & 	 -1.57698e+2 & 	 3.7e-7 & 	  3(33.0) & 	 3343.9& 	 -2.53730e+2 & 	 7.9e-7 & 	  3(47.3) & 	 1106.8\\ \hline 
		ASQN    & 	 -1.57698e+2 & 	 9.8e-7 & 	 33(23.3) & 	 4968.7& 	 -2.53730e+2 & 	 9.4e-7 & 	 23(25.0) & 	 1563.9\\ \hline 
		RQN    & 	 -1.57698e+2 & 	 4.1e-6 & 	 62 & 	 4946.7& 	 -2.53730e+2 & 	 9.7e-7 & 	 122 & 	 2789.4\\ \hline 
		& 	 \multicolumn{4}{c||}{al} & \multicolumn{4}{c|}{ptnio} \\ \hline 
		SCF            &   -3.52151e+2 &   7.4e+0 &   1000 &   4221.1 & 	 -9.25762e+2 & 	 1.9e-1 & 	 {1000} & 	 4461.9\\ \hline 
		GBB            & 	 -3.53707e+2 & 	 9.7e-7 & 	 1129 & 	 219.3& 	 -9.26927e+2 & 	 2.4e-6 & 	 10000 & 	 5627.2\\ \hline 
		ARNT    & 	 -3.53710e+2 & 	 5.9e-7 & 	 59(60.7) & 	 947.7& 	 -9.26927e+2 & 	 9.4e-7 & 	 104(129.6) & 	 7558.3\\ \hline 
		ASQN    & 	 -3.53710e+2 & 	 7.1e-7 & 	 94(47.3) & 	 1395.4& 	 -9.26927e+2 & 	 9.2e-7 & 	 153(69.6) & 	 12728.1\\ \hline 
		RQN    & 	 -3.53710e+2 & 	 1.8e-3 & 	 267 & 	 323.4& 	 -9.26925e+2 & 	 2.3e-4 & 	 380 & 	 924.4\\ \hline  
	\end{tabular} 
\end{table}

From Tables \ref{tab:ks} and \ref{tab:kshard}, we can see that SCF failed in ``graphene16'', ``graphene30'', ``al'', ``ptnio'' and ``c''. 
 We next  explain  why SCF fails by taking ``c'' and ``graphene16'' as examples. 
For the case ``c'', we obtain the same solution by using GBB, ARNT and ASQN. The
number of wanted wave functions are 2, i.e., $p=2$. With some abuse of notation, we
denote the final solution by $X = [x_1,x_2]$. Since $X$ satisfies the
first-order optimality condition, the columns of $X$ are also eigenvectors of $H(X)$ and the corresponding eigenvalues of
$H(X)$ are -1.8790, -0.6058. On the other hand, the smallest four eigenvalues of
$H(X)$ are -1.8790, -0.6577, -0.6058, -0.6058 and the corresponding
eigenvectors, denoted by $Y = [y_1,y_2,y_3,y_4]$. The  energies and
norms of Riemannian gradients of the different eigenvector pairs
$[x_1,x_2],~[y_1,y_2],~[y_1,y_3]$ and $[y_1,y_4]$ are $ (-5.3127, 9.96 \times
10^{-7}), (-5.2903, 3.07\times 10^{-1}), (-5.2937, 1.82 \times 10^{-1})$ and
$(-4.6759, 1.82 \times 10^{-1})$, respectively. Comparing the angles between $X$
and $Y$ shows that $x_1$ is nearly parallel to $y_1$ but $x_2$ lies in the
subspace spanned by $[y_3,y_4]$ other than $y_2$. Hence, when the SCF method
is used around $X$, the next point will jump to the subspace spanned by
$[y_1,y_2]$. This indicates the failure of the \textit{aufbau} principle, and
thus the failure of the SCF procedure. This is consistent with the observation
in the chemistry literature \cite{van2003density}, where sometimes the converged
solution may have a ``hole'' (i.e., unoccupied states) below the highest occupied energy level.

In the case ``graphene16'', we still obtain  the same solution from GBB, ARNT and
ASQN. The number of wave functions $p$ is 37. Let $X$ be the computed solution and
the corresponding eigenvalues of $H(X)$ be $d$. The smallest 37 eigenvalues and their
corresponding eigenvectors of  $H(X)$ are $g$ and $Y$. 
We find that the first 36 elements of $d$ and $g$ are almost the same up to a
machine accuracy, but the 37th element of $d$ and $g$ is 0.5821 and 0.5783,
respectively. The energies and norms of Riemannian gradients of $X$ and $Y$ are
$(-94.2613,8.65\times 10^{-7})$ and $(-94.2030,6.95\times 10^{-1})$, respectively. Hence, SCF
does not converge around the point $X$. 

In Tables \ref{tab:ks} and \ref{tab:kshard},  ARNT  usually converges in a few
iterations due to the usage of the second-order information. It is often the fastest
one in terms of time since the computational cost of two parts of the Hessian $\nabla^2 E_{\ks}$ has no significant difference. 
GBB also performs comparably well as ARNT. ASQN works reasonably well on most
problems. It takes more iterations than ARNT since the limit-memory approximation
often is not as good as the Hessian. Because the costs of solving the
subproblems of ASQN and ARNT are more or less the same, ASQN is not competitive
to ARNT. However, by taking advantage of the problem structures, ASQN is still
 better than RQN in terms of computational time and accuracy. Finally,  
 we show the convergence behaviors of these five  methods on the system
 ``glutamine'' in Figure \ref{fig:ks1}. 
Specifically, the error of the objective function values is defined as 
\[ \Delta E_{\ks}(X^k) = E_{\ks}(X^k) - E_{\min}, \]
where $E_{\min}$ be the minimum of the total energy attained
by all methods. 

\begin{figure}[htb]
	\subfigure[$\Delta E_{\ks}(X^k)$ versus iterations]{
		\includegraphics[width=0.45\textwidth,height=0.33\textwidth]
		{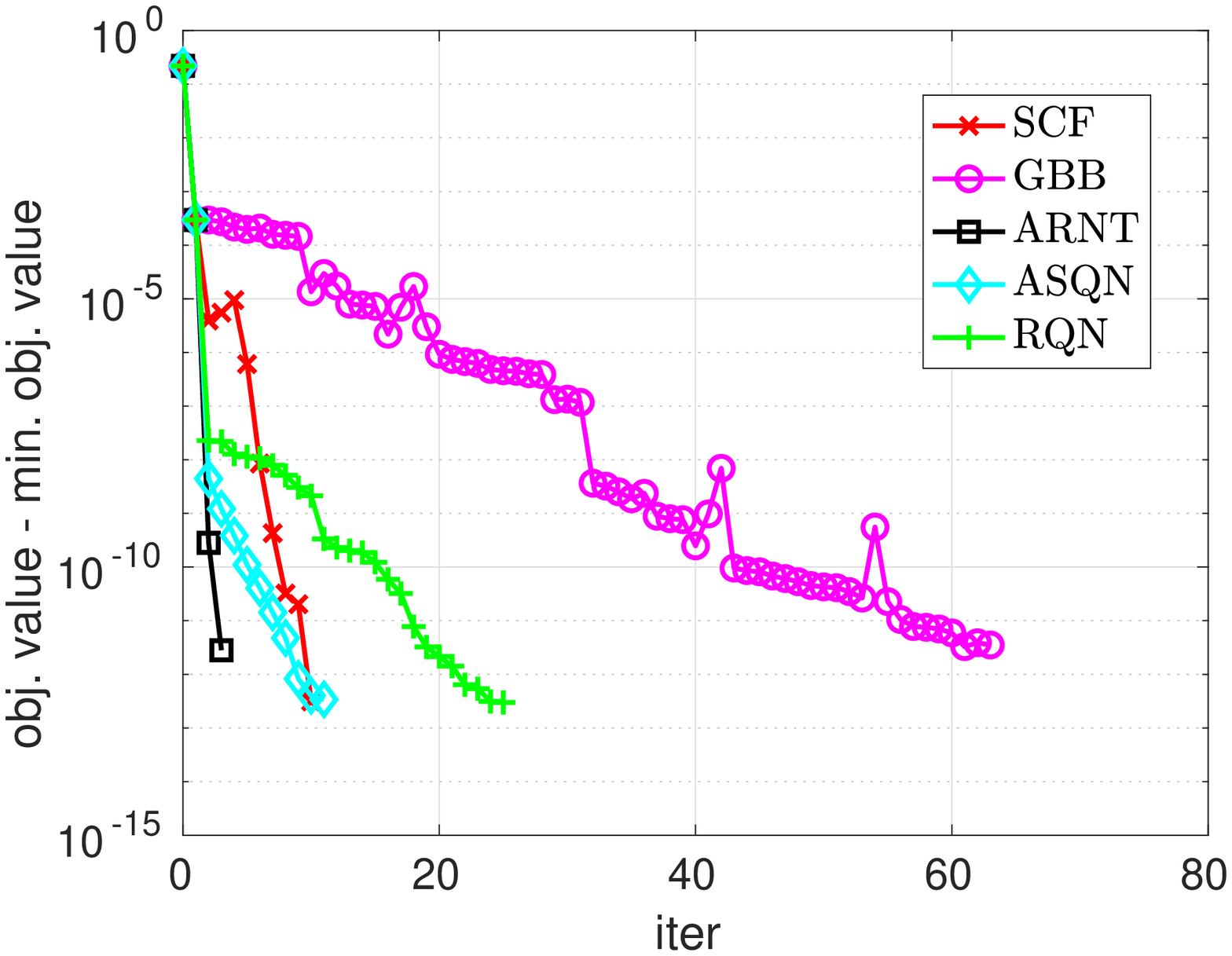}
	}
	\subfigure[$\Delta E_{\ks}(X^k)$ versus time]{
		\includegraphics[width=0.45\textwidth,height=0.33\textwidth]
		{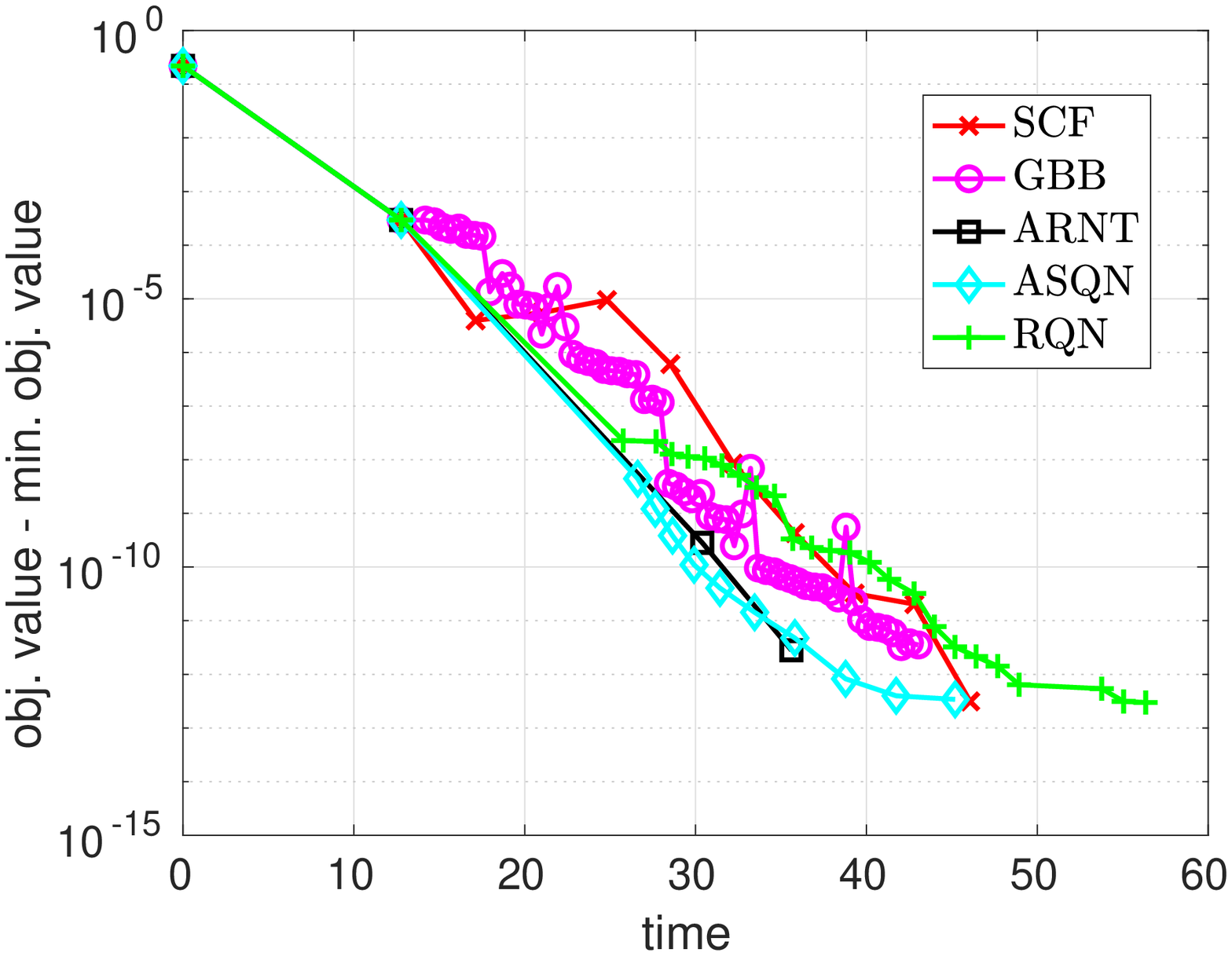}
	}
	
	\subfigure[$\|\grad E_{\ks}(X^k)\|_\Fsf$ versus iterations]{
		\includegraphics[width=0.45\textwidth,height=0.33\textwidth]
		{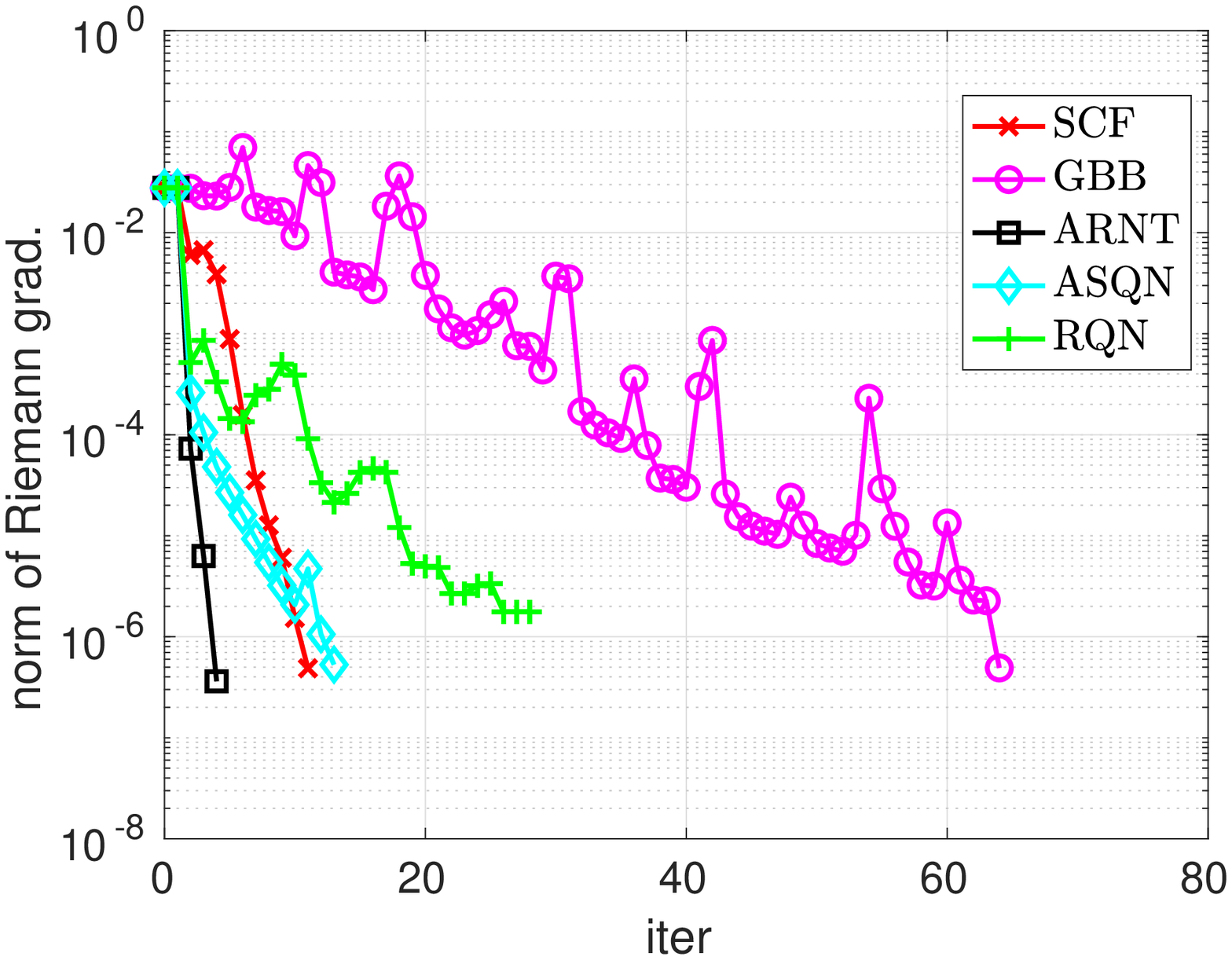}
	}
	\subfigure[$\|\grad E_{\ks}(X^k)\|_\Fsf$ versus time]{
		\includegraphics[width=0.45\textwidth,height=0.33\textwidth]
		{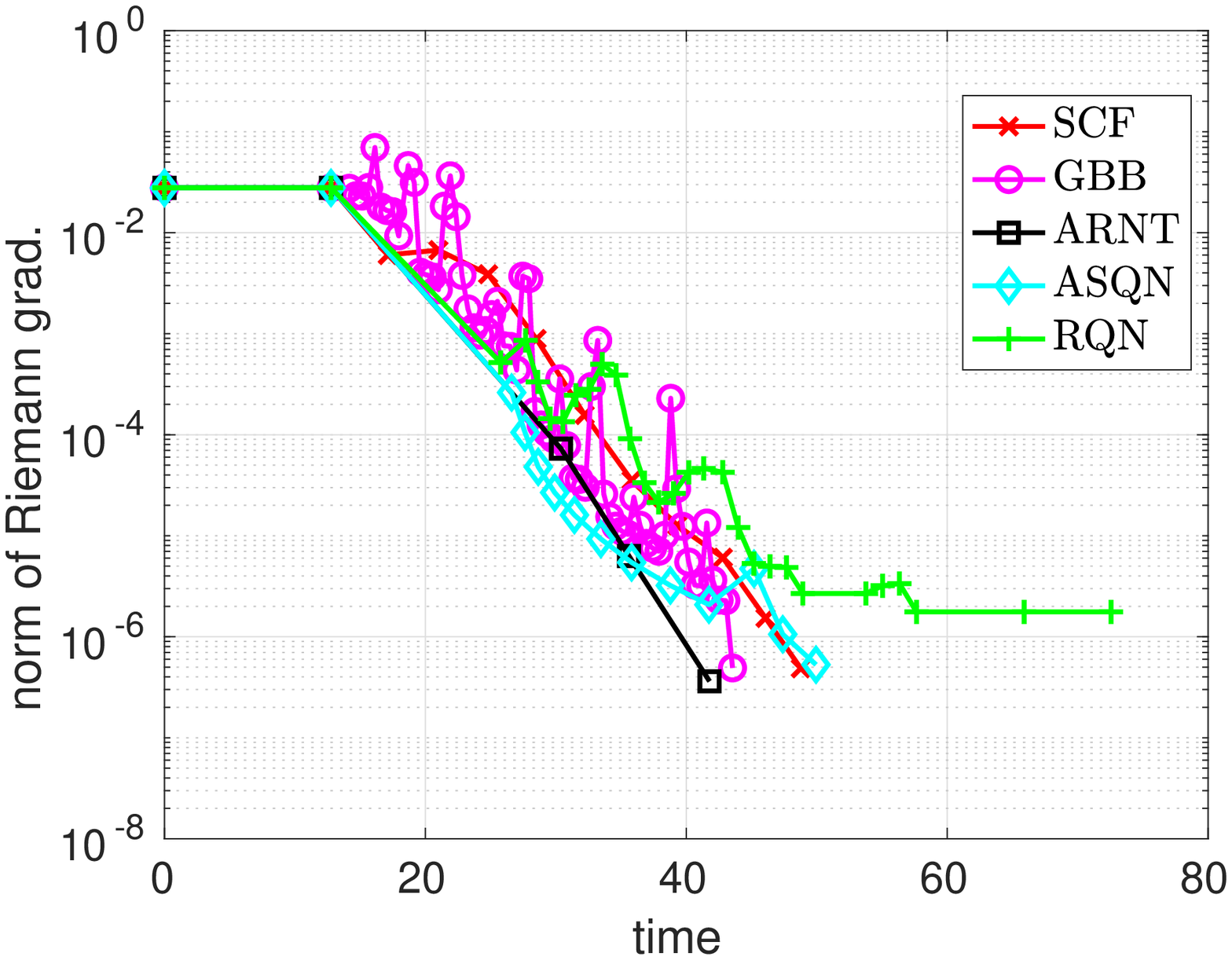}
	}  
	\caption{Comparisons of different algorithms on ``glutamine'' of KS total energy minimization. The first two points are the input and output of the initial solver GBB, respectively.}
	\label{fig:ks1}
\end{figure}

\begin{table}[!htp]
	\footnotesize
	\centering
	\caption{Numerical results on KS total energy minimization.} \label{tab:kshard} 
	\setlength{\tabcolsep}{1.5pt}
	\begin{tabular}{|c|c|c|c|c|}
		\hline 
		solver  & 	 fval & 	 nrmG & 	 its  & 	 time \\ \hline 
		\multicolumn{5}{|c|}{c} \\ \hline 
		SCF   & 	  -5.29296e+0 &   7.3e-3 &   1000 &   168.3 \\ \hline 
		GBB     & -5.31268e{+0} & 	 1.0e-6 & 	 3851 & 	 112.7\\ \hline 
		ARNT   & 	 -5.31268e{+0} & 	 5.7e-7 & 	 96(49.1) & 	 211.3\\ \hline 
		ASQN  & 	 -5.31268e{+0} & 	 6.7e-7 & 	 104(38.5) & 	 183.1\\ \hline 
		RQN    & 	 -5.31244e{+0}& 	 1.4e-3 & 	 73 & 	 10.8\\ \hline 
	\end{tabular} 
\end{table}

\subsection{Hartree-Fock total energy minimization} \label{sec:hf}
In this subsection, we compare the performance of three variants of Algorithm \ref{alg:ace} where the
subproblem is solved by SCF (ACE), the modified CG method (ARN) and by GBB
(GBBN), respectively, the Riemannian L-BFGS (RQN) method in Manopt
\cite{manopt}, and two variants of Algorithm
\ref{alg:struct-QN} with approximation \eqref{struct-B-app} (ASQN) and
approximation \eqref{eq:qn-hf} (AKQN).  Since the computation of the exact Hessian $\nabla^2 E_{\hf}$ is time-consuming, we do not present the results using the exact Hessian.  The limited-memory Nystr\"om
approximation \eqref{eq:nys} serves as an initial Hessian approximation in both
ASQN and AKQN. To compare the effectiveness of quasi-Newton approximation, we
set $\He(X^k)$ to be the limited-memory Nystr\"om approximation \eqref{eq:nys}
in \eqref{eq:qn-hf} and use the same framework as in Algorithm
\ref{alg:struct-QN}.  We should mention that the subspace refinement is not
used in ASQN and AKQN. Hence, only structured quasi-Newton iterations are
performed in them. The default parameters in RQN and GBB are used. For ACE,
GBBN, ASQN, AKQN and ARN, the subproblem is solved until the Frobenius-norm of
the Riemannian gradient is less than $0.1\min\{\|\grad f(X^k)\|_{{\Fsf}},
1\}$. We also use the adaptive strategy for choosing the maximal number of inner iterations of ARNT in \cite{hu2018adaptive} for GBBN, ASQN, AKQN and ARN. The settings of other parameters of ASQN, AKQN and ARN are the same to those in ARNT \cite{hu2018adaptive}.
For all algorithms, we generate a good initial guess by using GBB to solve the
corresponding KS total energy minimization problem (i.e., remove $E_{\f}$ part
from $E_{\hf}$ in the objective function) until a maximal number of iterations
2000 is reached or the Frobenius-norm of the Riemannian gradient is
smaller than $10^{-3}$. The maximal number of iterations for ACE, GBBN, ASQN,
ARN and AKQN is set to 200 while that of RQN is set to 1000.  

\begin{table}[!htp]
	\footnotesize
	\centering
	\caption{Numerical results on HF total energy minimization.} \label{tab:hf} 
	\setlength{\tabcolsep}{1.3pt}
	\begin{tabular}{|c|c|c|c|c||c|c|c|c|}
		\hline 
		solver  & 	 fval & 	 nrmG & 	 its  & 	 time  & 	 fval & 	 nrmG & 	 its  & 	 time \\ \hline 
		& 	 \multicolumn{4}{c||}{alanine} & \multicolumn{4}{c|}{c12h26} \\ \hline
		ACE            & 	 -6.61821e+1 & 	 3.8e-7 & 	 11(3.0) & 	 261.7& 	 -8.83756e+1 & 	 3.9e-7 & 	  8(2.9) & 	 259.7\\ \hline 
		GBBN   & 	 -6.61821e+1 & 	 1.0e-6 & 	 11(17.4) & 	 268.8& 	 -8.83756e+1 & 	 4.9e-4 & 	 200(68.7) & 	 11839.8\\ \hline 
		ARN    & 	 -6.61821e+1 & 	 9.5e-7 & 	 10(13.7) & 	 206.6& 	 -8.83756e+1 & 	 4.9e-4 & 	 200(2.4) & 	 4230.3\\ \hline 
		ASQN    & 	 -6.61821e+1 & 	 9.1e-7 & 	  7(14.1) & 	 169.6& 	 -8.83756e+1 & 	 2.1e-7 & 	  7(12.6) & 	 234.1\\ \hline 
		AKQN   & 	 -6.61821e+1 & 	 4.8e-7 & 	 31(7.5) & 	 530.2& 	 -8.83756e+1 & 	 4.9e-7 & 	 29(7.6) & 	 871.2\\ \hline 
		RQN    & 	 -6.61821e+1 & 	 1.9e-6 & 	 76 & 	 1428.5& 	 -8.83756e+1 & 	 1.3e-3 & 	 45 & 	 3446.3\\ \hline 
		& 	 \multicolumn{4}{c||}{ctube661} & \multicolumn{4}{c|}{glutamine} \\ \hline 
		ACE            & 	 -1.43611e+2 & 	 9.2e-7 & 	  8(2.8) & 	 795.0& 	 -1.04525e+2 & 	 3.9e-7 & 	 10(3.0) & 	 229.6\\ \hline 
		GBBN   & 	 -1.43611e+2 & 	 6.5e-7 & 	 10(26.3) & 	 1399.2& 	 -1.04525e+2 & 	 8.4e-7 & 	 11(13.3) & 	 256.9\\ \hline 
		ARN    & 	 -1.43611e+2 & 	 6.0e-7 & 	  9(14.1) & 	 832.7& 	 -1.04525e+2 & 	 8.8e-7 & 	 10(9.5) & 	 209.5\\ \hline 
		ASQN    & 	 -1.43611e+2 & 	 2.0e-7 & 	  8(13.2) & 	 777.1& 	 -1.04525e+2 & 	 1.5e-7 & 	  8(10.1) & 	 182.9\\ \hline 
		AKQN   & 	 -1.43611e+2 & 	 6.1e-7 & 	 17(10.3) & 	 1502.0& 	 -1.04525e+2 & 	 9.1e-7 & 	 25(6.0) & 	 515.7\\ \hline 
		RQN    & 	 -1.43611e+2 & 	 7.2e-6 & 	 59 & 	 6509.0& 	 -1.04525e+2 & 	 2.9e-6 & 	 57 & 	 1532.8\\ \hline 
		& 	 \multicolumn{4}{c||}{graphene16} & \multicolumn{4}{c|}{graphene30} \\ \hline 
		ACE            & 	 -1.01716e+2 & 	 7.6e-7 & 	 13(3.4) & 	 367.0 & -1.87603e+2 & 	 8.6e-7 & 	 58(4.2) & 	 14992.0 	\\ \hline 
		GBBN   & 	 -1.01716e+2 & 	 4.2e-7 & 	 14(42.1) & 	 659.0 & -1.87603e+2 & 	 8.9e-7 & 	 29(72.2) & 	 19701.8                 \\ \hline 
		ARN    & 	 -1.01716e+2 & 	 4.5e-7 & 	 14(23.0) & 	 403.6 & -1.87603e+2 & 	 9.0e-7 & 	 45(35.6) & 	 14860.6\\ \hline 
		ASQN    & 	 -1.01716e+2 & 	 4.9e-7 & 	 11(20.2) & 	 357.5 & -1.87603e+2 & 	 7.6e-7 & 	 15(26.5) & 	 6183.0 \\ \hline 
		AKQN   & 	 -1.01716e+2 & 	 7.9e-7 & 	 49(15.1) & 	 1011.0 & -1.87603e+2 & 	 8.0e-7 & 	 39(12.3) & 	 9770.7\\ \hline 
		RQN    & 	 -1.01716e+2 & 	 1.0e-3 & 	 74 & 	 2978.9 & -1.87603e+2 & 	 1.5e-5 & 	 110 & 	 39091.0 \\ \hline 
		& 	 \multicolumn{4}{c||}{pentacene} & \multicolumn{4}{c|}{gaas} \\ \hline 
		ACE            & 	 -1.39290e+2 & 	 6.2e-7 & 	 13(3.0) & 	 1569.5& 	 -2.93496e+2 & 	 8.8e-7 & 	 29(2.9) & 	 343.8\\ \hline 
		GBBN   & 	 -1.39290e+2 & 	 8.2e-7 & 	 16(23.0) & 	 2620.2& 	 -2.93496e+2 & 	 9.3e-7 & 	 34(35.3) & 	 659.3\\ \hline 
		ARN    & 	 -1.39290e+2 & 	 7.2e-7 & 	 15(12.2) & 	 1708.1& 	 -2.93496e+2 & 	 9.6e-7 & 	 31(20.4) & 	 468.7\\ \hline 
		ASQN    & 	 -1.39290e+2 & 	 1.9e-7 & 	  9(14.3) & 	 1168.1& 	 -2.93496e+2 & 	 3.3e-7 & 	 10(28.0) & 	 199.5\\ \hline 
		AKQN   & 	 -1.39290e+2 & 	 5.4e-7 & 	 29(8.5) & 	 3458.4& 	 -2.93496e+2 & 	 4.6e-7 & 	 22(18.4) & 	 347.1\\ \hline 
		RQN    & 	 -1.39290e+2 & 	 2.4e-6 & 	 73 & 	 11363.8& 	 -2.93496e+2 & 	 1.0e-6 & 	 126 & 	 2154.1\\ \hline 
		& 	 \multicolumn{4}{c||}{si40} & \multicolumn{4}{c|}{si64} \\ \hline 
		ACE            & 	 -1.65698e+2 & 	 9.2e-7 & 	 29(4.5) & 	 30256.4& 	 -2.67284e+2 & 	 9.8e-7 & 	  9(2.9) & 	 6974.3\\ \hline 
		GBBN   & -1.65698e+2 & 	 8.6e-7 & 	 24(43.9) & 	 34692.4 & -2.67284e+2 & 5.3e-7 & 14(27.0) & 11467.9 \\ \hline
		ARN    & 	 -1.65698e+2 & 	 8.0e-7 & 	 22(22.1) & 	 21181.3& 	 -2.67284e+2 & 	 7.7e-7 & 	 12(18.6) & 	 9180.7\\ \hline 
		ASQN    & 	 -1.65698e+2 & 	 2.8e-7 & 	 12(37.8) & 	 15369.5& 	 -2.67284e+2 & 	 3.0e-7 & 	  8(21.9) & 	 6764.7\\ \hline
		AKQN   & -1.65698e+2 & 	 9.2e-7 & 	 87(7.9) & 	    89358.8 &-2.67284e+2 & 7.1e-7 & 24(18.8) & 33379.0 \\ \hline
		RQN    & -1.65698e+2 &      6.1e-6 &       156  &      181976.8  &-2.67284e+2 & 8.4e-7 & 112 & 115728.8 \\\hline
	\end{tabular} 
\end{table}

A detailed summary
of computational results is reported in Table \ref{tab:hf}. We see that ASQN performs best
among all the algorithms in terms of both the number of iterations and time,
especially in the systems: ``alanine'', ``graphene30'', ``gaas'' and ``si40''. Usually,
algorithms takes fewer iterations if more parts in the Hessian are preserved.
Since the computational cost of the Fock energy dominates that of the KS part,
algorithms using fewer outer iterations consumes less time to converge. Hence,
ASQN is faster than AKQN. Comparing with ARN and RQN, we see that ASQN benefits from
our quasi-Newton technique. Using a scaled identity matrix as the initial
 guess, RQN takes many more iterations than our algorithms which use the adaptive compressed form
 of the hybrid exchange operator. ASQN is two times faster than ACE in
 ``graphene30'' and ``si40''. Finally,  
  the convergence behaviors of these six  methods on the system
 ``glutamine'' in Figure \ref{fig:hf1}, where $\Delta E_{\hf}(X^k)$ is defined
 similarly as the KS case.  In summary, algorithms utilizing the quasi-Newton
 technique combining with the Nystr\"om approximation is often able to give better performance.   

\begin{figure}[htb]
	\subfigure[$\Delta E_{\hf}(X^k)$ versus iterations]{
		\includegraphics[width=0.45\textwidth,height=0.33\textwidth]
		{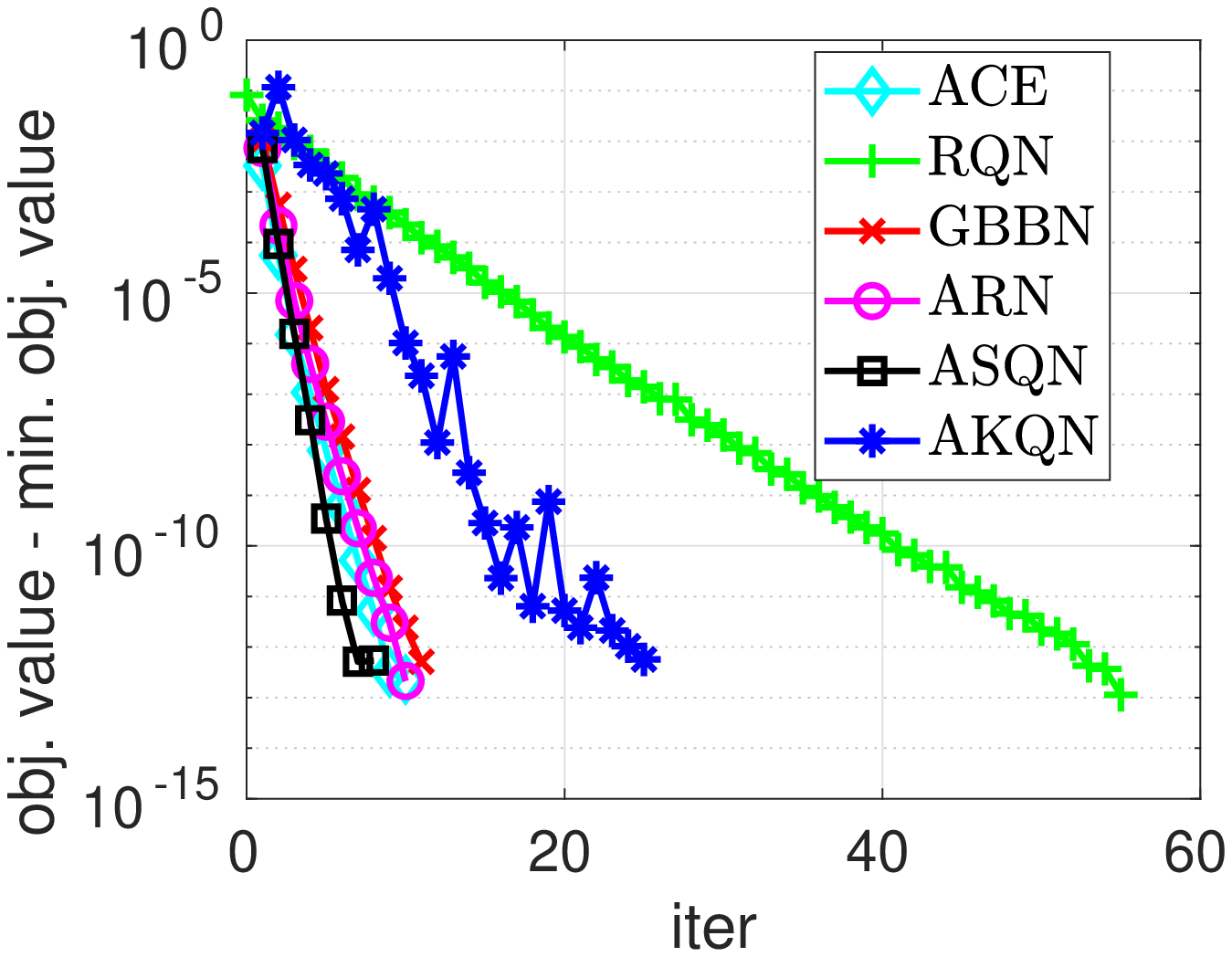}
	}
	\subfigure[$\Delta E_{\hf}(X^k)$ versus time]{
		\includegraphics[width=0.45\textwidth,height=0.33\textwidth]
		{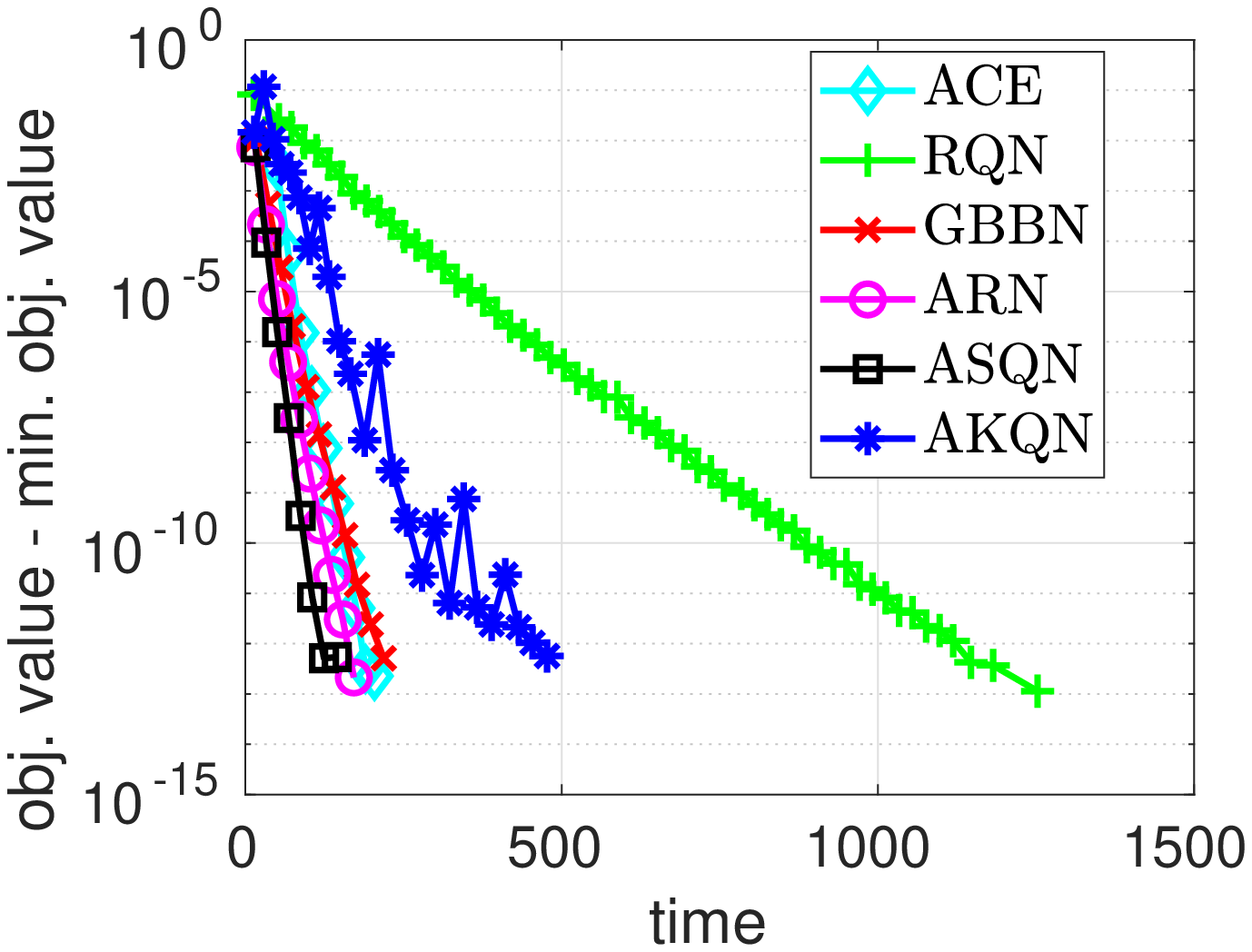}
	}
	
	\subfigure[$\|\grad E_{\hf}(X^k)\|_\Fsf$ versus iteration]{
		\includegraphics[width=0.45\textwidth,height=0.33\textwidth]
		{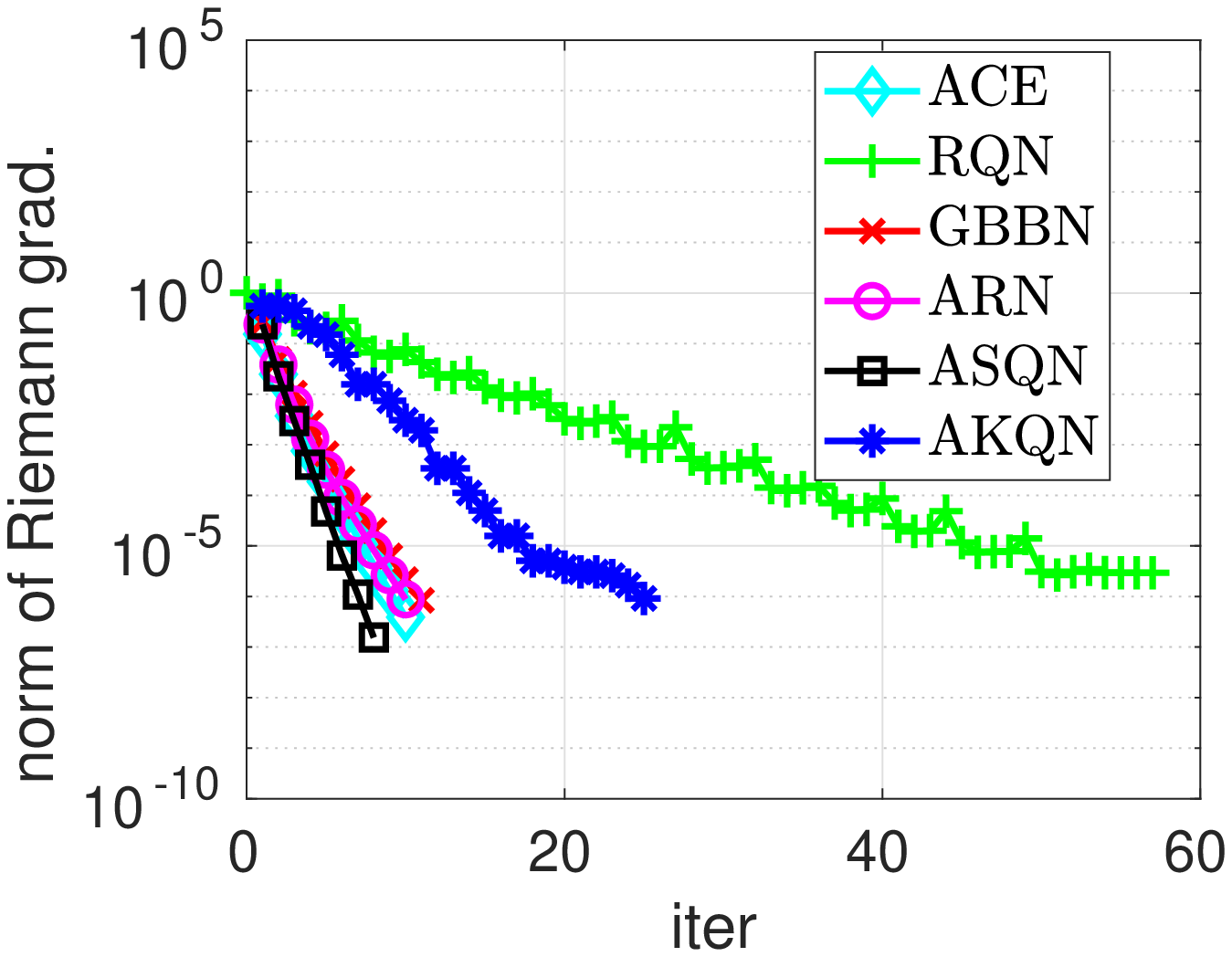}
	}
	\subfigure[$\|\grad E_{\hf}(X^k)\|_\Fsf$ versus time]{
		\includegraphics[width=0.45\textwidth,height=0.33\textwidth]
		{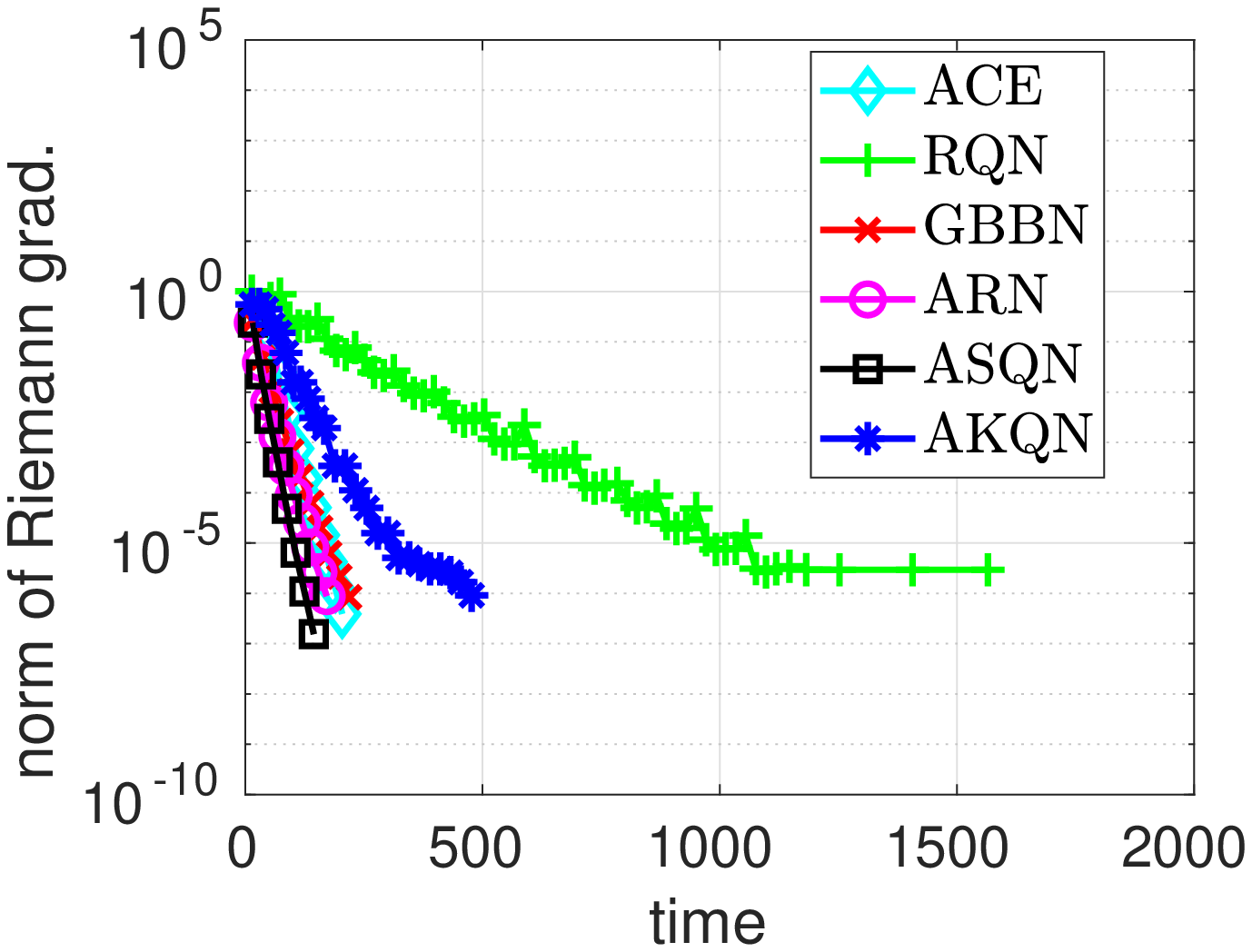}
	}  
	\caption{Comparisons of different algorithms on ``glutamine'' of HF total energy minimization.}
	\label{fig:hf1}
\end{figure}

\section{Conclusion}
We present a structured quasi-Newton method for optimization with orthogonality
constraints. Instead of approximating the full Riemannian Hessian directly, we
construct an approximation to the Euclidean Hessian and a regularized subproblem
using this approximation while the orthogonality constraints are kept. By solving
the subproblem inexactly, the global and local q-superlinear convergence can be
guaranteed under certain assumptions. Our structured quasi-Newton method also takes
advantage of the structure of the objective function if some parts are much
more expensive to be evaluated than other parts. Our numerical experiments on
the linear eigenvalue problems, KSDFT and HF total energy minimization
demonstrate that our structured quasi-Newton algorithm is very competitive with the state-of-art algorithms. 

The performance of the quasi-Newton methods can be further improved in several
 perspectives. For example, finding a better initial quasi-Newton matrix than
 the Nystr\"om approximation and developing a better quasi-Newton approximation
 than the LSR1 technique. Our technique can also be extended to the general
 Riemannian optimization with similar structures.

\bibliographystyle{siamplain}
\bibliography{optimization}

\end{document}